\documentclass[11pt,letterpaper]{amsart}

\usepackage[letterpaper,margin=1.3in,footskip=0.70in]{geometry}
\usepackage{microtype}
\usepackage{amsmath,amssymb,amsthm}
\usepackage{mathtools}
\usepackage[dvipsnames]{xcolor}
\usepackage{enumitem}
\setlist{noitemsep}
\usepackage{mathrsfs}
\usepackage{tikz-cd}
\usepackage{array}
\usepackage[pdfusetitle,colorlinks]{hyperref}
\hypersetup{citecolor=black,linkcolor=black,urlcolor=black}
\usepackage[capitalise,noabbrev]{cleveref}
\crefformat{equation}{\ensuremath{(#2#1#3)}}
\crefmultiformat{equation}{\ensuremath{(#2#1#3)}}{ and~\ensuremath{(#2#1#3)}}{, \ensuremath{(#2#1#3)}}{, and~\ensuremath{(#2#1#3)}}
\numberwithin{figure}{section}
\numberwithin{equation}{subsection}
\newtheorem{theorem}[figure]{Theorem}
\newtheorem{lemma}[figure]{Lemma}
\newtheorem{corollary}[figure]{Corollary}
\newtheorem{proposition}[figure]{Proposition}
\theoremstyle{definition}
\newtheorem{definition}[figure]{Definition}
\newtheorem{notation}[figure]{Notation}
\theoremstyle{definition}
\newtheorem{remark}[figure]{Remark}
\theoremstyle{definition}
\newtheorem{example}[figure]{Example}
\theoremstyle{definition}
\newtheorem{construction}[figure]{Construction}

\usepackage[all]{xy}
\usepackage{relsize}
\usepackage[bbgreekl]{mathbbol}
\usepackage{amsfonts}
\DeclareSymbolFontAlphabet{\mathbb}{AMSb} 
\DeclareSymbolFontAlphabet{\mathbbl}{bbold}
\newcommand{\Prism}{{\mathlarger{\mathbbl{\Delta}}}}

\usepackage{lipsum}
\usepackage{fancyhdr}
\numberwithin{figure}{section}
\numberwithin{equation}{section}

\theoremstyle{definition}
\theoremstyle{definition}
\theoremstyle{definition}

\theoremstyle{definition}

\theoremstyle{definition}
\newtheorem{question}[figure]{Question}

\mathchardef\mhyphen="2D

\DeclareMathOperator{\Spec}{Spec}

\usepackage{relsize}
\usepackage[bbgreekl]{mathbbol}
\usepackage{amsfonts}
\DeclareSymbolFontAlphabet{\mathbb}{AMSb} 
\DeclareSymbolFontAlphabet{\mathbbl}{bbold}

\newcommand{\FGauge}{F\text{-}\mathrm{Gauge}}

\newcommand\restr[2]{{\left.\kern-\nulldelimiterspace#1\vphantom{\big|}\right|_{#2}}}
\newcommand{\suchthat}{\;\ifnum\currentgrouptype=16 \middle\fi\vert\;}

\newcolumntype{C}[1]{>{\centering\arraybackslash}p{#1}}

\title{Zeta function of $F$-gauges and special values}

\author{Shubhodip Mondal}

\address[Shubhodip Mondal]{University of British Columbia, Vancouver BC, Canada}
\email{smondal@math.ubc.ca}

\begin{document}

\begin{abstract}
In 1966, Tate proposed the Artin--Tate conjectures, which expresses special values of zeta function associated to surfaces over finite fields. Conditional on the Tate conjecture, Milne--Ramachandran formulated and proved similar conjectures for smooth proper schemes over finite fields. The formulation of these conjectures already relies on other unproven conjectures. In this paper, we give an unconditional formulation of these conjectures for dualizable $F$-gauges over finite fields and prove them. In particular, our results also apply unconditionally to smooth proper varieties over finite fields. A key new ingredient is the notion of ``stable Bockstein characteristic" that we introduce. Our proof uses techniques from the stacky approach to $F$-gauges recently introduced by Drinfeld and Bhatt--Lurie and the author's recent work on Dieudonn\'e theory using $F$-gauges.

\end{abstract}

\maketitle
\tableofcontents
\section{Introduction}
Let $p$ be a fixed prime and $q$ be a power of $p$. For a smooth proper variety $X$ over the finite field $\mathbb{F}_q$, one defines the zeta function of $X$ denoted as $$Z(X,t) := \mathrm{exp} \sum_{m \ge 1} \frac{|X(\mathbb{F}_{q^m})|t^m}{m}.$$ In 1966, Tate \cite{tate} proposed the Artin--Tate conjectures, which says that for a surface $X$ over $\mathbb{F}_q$, the Brauer group $\mathrm{Br}(X)$ is finite; conditional on this conjecture, they further conjectured an expression for certain special value of $Z(X,t)$ that involves $|\mathrm{Br}(X)|$ as well as the N\'eron--Severi group of $X$. Milne \cite{Mana} proved that for a surface $X$ over a finite field of odd characteristic, the conjecture regarding finiteness of $\mathrm{Br}(X)$ implies the conjecture regarding special values. Assuming the Tate conjecture, in \cite{MRjams}, Milne and Ramachandran proved an expression regarding special values of zeta functions associated to smooth projective varieties over finite fields. Further, conditional on the Tate conjecture, in \cite{MRpre}, they formulated certain conjectures regarding special values of zeta functions that one may expect to attach to a motive over a finite field.  In \cite{MRp}, Milne--Ramachandran studied this problem for zeta functions associated to objects in $D^b_{{c}}(\mathcal{R})$, where $\mathcal{R}$ is the Raynaud ring \cite{IR}. Roughly speaking, $D^b_{{c}}(\mathcal{R})$ may be thought of as the $p$-adic realization of the category of motives over finite fields. Conditional on certain assumptions, they formulated and proved expressions involving special values of zeta functions. In all of these work, even the formulation of the expressions regarding special values depend on other unproven conjectures or assumptions.

In this paper, we formulate and prove an \emph{unconditional} expression regarding special values of zeta functions associated to smooth proper varieties over finite fields. In particular, we remove the dependence on unproven conjectures or hypotheses that appeared in the previous results in this area. In fact, we work more generally with the category of $F$-gauges over $\mathbb{F}_q$ as defined and studied by Mazur, Ekedahl, Fontaine--Jannsen, and others \cite{Mazur, eke, FJ13}. We prefer to work with $F$-gauges instead of $D^b_{c}(\mathcal{R})$ because the former notion is more flexible and it also admits a generalization to mixed characteristic, as demonstrated in the recent work of Drinfeld \cite{prismatization} and Bhatt--Lurie \cite{fg} involving certain prismatization stacks. In particular, they define a stack denoted as $(\mathrm{Spec}\, \mathbb{F}_q)^{\mathrm{syn}}$ such that the category of $F$-gauges over $\mathbb{F}_q$ is equivalent to the derived category of quasicoherent sheaves on $(\mathrm{Spec}\, \mathbb{F}_q)^{\mathrm{syn}}$.

For a dualizable $F$-gauge $M$, in \cref{zetafu}, we define a zeta function $Z(M,t)$ attached to it. We set $\zeta (M, s):= Z(M, q^{-s})$. One of the key new ingredients in our work is the notion of \textit{stable Bockstein characteristic}, which we introduce in \cref{stablebock}. Roughly speaking, this gives a way to ``measure" certain infinite type complexes. Using this notion, for every $r \in \mathbb{Z}$, we introduce an invariant $\mu_{\mathrm{syn}}(M, r)$ to the $F$-gauge $M$ (see \cref{introdef}). Our main result is as follows.

\begin{theorem}\label{mainthm}
Let $M$ be a dualizable $F$-gauge over $\mathbb{F}_q$. Let $r \in \mathbb{Z}$. Suppose that $\rho$ is the order of the zero of $\zeta (M, s)$ at $s=r$. Then 
$$\left| \lim_{s \to r} \frac{\zeta (M,s)}{(1-q^{r-s})^\rho}\right| _p = \frac{1}{\mu_{\mathrm{syn}}(M,r)q^{\chi(M,r)}}.$$\end{theorem}
In the above, $\chi(M,r)$ is defined to be $$\chi(M,r):= \sum_{\substack{i,j \in \mathbb{Z}, \\ i \le r}} (-1)^{i+j} (r-i)h^{i,j}(M),$$where $h^{i,j}(M)$ denotes the Hodge--numbers of the $F$-gauge $M$, as defined in \cref{hodgerea}. $|\cdot|_p$ denotes the normalized $p$-adic norm.

 As the notion of $\mu_{\mathrm{syn}}(M,r)$ is crucial in the unconditional formulation of \cref{mainthm}, we include some remarks and motivations related to it. Roughly speaking, $\mu_{\mathrm{syn}}(M,r)$ ``measures" the size of $R\Gamma_{\mathrm{syn}}(M, \mathbb{Z}_p(r))$, where the latter denotes syntomic cohomology of the dualizable $F$-gauge $M$ of weight $r$ (see \cref{sad}). Note that while $R\Gamma(M, \mathbb{Z}_p(r))[1/p]$ is a complex of finite dimensional $\mathbb{Q}_p$-vector space, the dimension (i.e., Euler characteristic) of $R\Gamma(M, \mathbb{Z}_p(r))[1/p]$ as a $\mathbb{Q}_p$-vector space is always zero (\cref{dim0}). Thus, dimension is not the desired notion in our context. Furthermore, the cohomology groups of $R\Gamma_{\mathrm{syn}}(M, \mathbb{Z}_p(r))$ are typically not finite abelian groups, which makes it difficult to measure its size. In the remark below, we discuss a similar, but different scenario.

\begin{remark}
Let $(\mathrm{Spf}\, \mathbb{Z}_p)^{\mathrm{syn}}$ be the ``syntomification" of $\mathrm{Spf}\,\mathbb{Z}_p$ as in \cite{fg, prismatization}. A coherent sheaf $M$ on $(\mathrm{Spf}\, \mathbb{Z}_p)^{\mathrm{syn}}$ may be called a coherent prismatic $F$-gauge over $\mathbb{Z}_p$. Let $G_{\mathbb{Q}_p}$ denote $\mathrm{Gal}(\overline{\mathbb Q}_p/ {\mathbb Q}_p ).$ With $M$, one can canonically attach a finite free $\mathbb{Z}_p$-module $T$ equipped with a $G_{\mathbb{Q}_p}$-action such that $V:=T[1/p]$ is a crystalline Galois representation. The syntomic cohomology of $M$ (in weight $0$) is defined to be $R\Gamma((\mathrm{Spf}\, \mathbb{Z}_p)^{\mathrm{syn}}, M)$. We will explain how to attach a \emph{numerical measure} to the syntomic cohomology group $H^1 ((\mathrm{Spf}\, \mathbb{Z}_p)^{\mathrm{syn}}, M).$ As proven in \cite{fg}, there is an isomorphism $H^1 ((\mathrm{Spf}\, \mathbb{Z}_p)^{\mathrm{syn}}, M)[1/p] \simeq H^1_f (G_{\mathbb{Q}_p}, V),$ where the latter denotes the Bloch--Kato Selmer groups defined in \cite{bloka}. Let $\mathbf{D}_{\mathrm{dR}}(V):= (V \otimes_{\mathbb{Q}_p} B_{\mathrm{dR}})^{\mathbb{Q}_p}$ \cite{fon}. There is a natural filtration on $\mathbf{D}_{\mathrm{dR}}(V)$ which we denote by $\mathrm{Fil}^*\mathbf{D}_{\mathrm{dR}}(V)$. Now in \cite{bloka}, Bloch--Kato defined an exponential map
$$\mathbf{D}_{\mathrm{dR}}(V)/ \mathrm{Fil}^0\mathbf{D}_{\mathrm{dR}}(V) \to H^1_{f}(G_{\mathbb{Q}_p}, V).$$ Under the assumption that the local $L$-function associated to $V$ does not vanish at $1$, it follows that the above map is an isomorphism. Via this isomorphism, one obtains a measure on $H^1_{f}(G_{\mathbb{Q}_p}, V)$ by using the Haar measure on $\mathbf{D}_{\mathrm{dR}}(V)/ \mathrm{Fil}^0\mathbf{D}_{\mathrm{dR}}(V)$ having total measure $1$. Now considering the measure of the image of the map $H^1 ((\mathrm{Spf}\, \mathbb{Z}_p)^{\mathrm{syn}}, M) \to 
 H^1_{f}(G_{\mathbb{Q}_p}, V)$, one can define a numerical invariant associated to $H^1((\mathrm{Spf}\, \mathbb{Z}_p)^{\mathrm{syn}}, M).$ 
\end{remark}

For a dualizable $F$-gauge $M$ over $\mathbb{F}_q$, one cannot apply any of the above techniques involving $p$-adic Galois representations or Haar measure to measure $R\Gamma_{\mathrm{syn}}(M, \mathbb{Z}_p(r)):= R\Gamma ((\mathrm{Spec}\, \mathbb{F}_q)^{\mathrm{syn}}, M\left\{r\right \})$ (\cref{sad}). Instead, we introduce some new techniques inspired by algebraic topology to ``measure" certain infinite type complexes. We work with an object $M \in D^b(\mathbb{Z}_p)$, equipped with an endomorphism $\theta: M \to M.$ In this situation, we define a certain chain complex that we call \emph{Bockstein complex} (see \cref{bockst}), and denote it by $\mathrm{Bock}^\bullet (M, \theta).$ In \cref{breka}, we discuss an interpretation of this complex in terms of the Beilinson $t$-structure on filtered derived category. After developing some formal properties of Bockstein complexes, in \cref{stablebock}, we show that under certain mild assumptions, the cohomology groups of $\mathrm{Bock}^\bullet (M, \theta^r)$ are finite length $\mathbb{Z}_p$-modules for all $r \gg 0$, and thus the (length) Euler characteristic $\chi^{l}(\mathrm{Bock}^\bullet (M, \theta^r))$ is well-defined. Further, in \cref{stabock}, we show that $$\lim_{r \to \infty} \frac{\chi^l (\mathrm{Bock}^\bullet(M, \theta^r))}{r}$$ exists and is an integer. We call this integer the \emph{stable Bockstein characteristic} and denote it by $\chi^l_{s} (\mathrm{Bock}^\bullet(M, \theta))$ We apply this machinery to measure $R\Gamma_{\mathrm{syn}}(M, \mathbb{Z}_p(r))$ as follows.

\begin{definition}\label{introdef}
For a dualizable $F$-gauge $M$ over $\mathbb{F}_q$, we define $$\mu_{\mathrm{syn}}(M, r):= p^{{\chi^l_s(\mathrm{Bock}^\bullet (R\Gamma_{\mathrm{syn}}(\overline{M}, \mathbb{Z}_p(r)), \gamma-1))}}.$$ Here, $\overline{M}$ denotes the base change of $M$ to $(\mathrm{Spec}\, \overline{\mathbb F}_q)^{\mathrm{syn}}$ and $\gamma$ denotes the Galois action on $R\Gamma_{\mathrm{syn}}(\overline{M}, \mathbb{Z}_p(r))$.
\end{definition}

Some work is needed to show that the formalism surrounding stable Bockstein characteristic is applicable to the pair $(R\Gamma_{\mathrm{syn}}(\overline{M}, \mathbb{Z}_p(r)), \gamma-1))$. This is carried out in \cref{someres}, where the notion of isocrystals play an important role. Below, we record some consequences of \cref{mainthm}.

\begin{remark}\label{katz}
 For a smooth proper variety $X$ over $\mathbb{F}_q$ one can associate a dualizable $F$-gauge $M(X)$ over $\mathbb{F}_q$. This can be explained in the stacky language as follows. One has a map $f: X^{\mathrm{syn}} \to (\mathrm{Spec}\,\mathbb{F}_q)^\mathrm{syn}$. We define $M(X):=Rf_* \mathcal{O}$. It follows from \cite[Thm.~3.3.5]{fg} and the work of Katz--Messing \cite{katzmes} that $Z(M(X), t)$ agrees with the usual zeta function $Z(X,t)$ of $X$. We denote $\mu_{\mathrm{syn}}(X, r):=\mu_{\mathrm{syn}}(M(X), r)$. As a consequence of \cref{mainthm}, we obtain that 
\begin{equation}\label{specia}
\left| \lim_{s \to r} \frac{\zeta (X,s)}{(1-q^{r-s})^\rho}\right| _p = \frac{1}{\mu_{\mathrm{syn}}(X,r)q^{\chi(X,r)}}.    
\end{equation}Here, $\chi(X,r):= \chi(M(X),r)$ is the same as $\sum_{\substack{i,j \in \mathbb{Z}, \\ i \le r}} (-1)^{i+j} (r-i)h^{i,j}$, where $h^{i,j}$ denotes the Hodge--numbers of $X$.
\end{remark}

\begin{remark}
 If one assumes certain finiteness conjectures as in \cite{MRpre}, or the conjecture that $q^{-r}$ is a semisimple eignenvalue of Frobenius on $H^i _{\mathrm{crys}}(X) \otimes \mathbb{Q}$ for all $i$ as in \cite{MRp}, then it follows that the cohomology groups of the complex (defined by multiplying with a canonical class in $H^1(\mathrm{Spec}\, \mathbb{F}_q, \mathbb{Z}_p)$)
$$\ldots  H^{i-1} (X, \mathbb{Z}_p(r)) \to H^i (X, \mathbb{Z}_p(r)) \to H^{i+1}(X, \mathbb{Z}_p(r)) \to \ldots $$are finite abelian groups. Milne--Ramachandran defined $\chi^\times (X, \mathbb{Z}_p(r))$ to be alternating product of the sizes of these abelian groups. Under these conjectures, the quantity $\mu_{\mathrm{syn}}(X,r)$ we defined can be directly shown to agree with $\chi^\times (X, \mathbb{Z}_p(r))$. Thus \cref{mainthm} combined with \cref{katz} recovers the work of Milne--Ramachandran for smooth proper varieties. 
\end{remark}

\begin{remark}
In the case of a smooth, proper, geometrically connected surface $X$ over $\mathbb{F}_q$, equation \cref{specia} gives an expression for the special value of $\zeta(X, s)$ at $s=1$ that circumvents the Artin--Tate conjecture regarding finiteness of $\mathrm{Br}(X)$. In the case of surfaces, $\mu_{\mathrm{syn}}(X,1)$ is related to the Brauer group and the N\'eron--Severi group. See \cref{surface}.
\end{remark}{}

\begin{remark}
 Under the semisimplicity assumptions as in \cite{MRp}, \cref{mainthm} also recovers the main theorem of Milne--Ramachandran in loc.~cit. This can be directly deduced from a structural result proven by Ekedahl \cite[Thm.~5.3]{eke}, where he showed that $D^b_{c}(\mathcal{R})$ embeds fully faithfully in the category of $F$-gauges.  We are therefore able to recover \cite{MRp} by entirely circumventing the formalism of de Rham--Witt complexes (see \cite{Luc1, IR}) and the detailed study of certain numerical invariants that appear in loc.~cit.
\end{remark}
 The method of the proof of \cref{mainthm} is very different from the previous approaches. Let us explain the key new ingredients in our proof. We use techniques from the stacky approach to $F$-gauges due to Drinfeld \cite{prismatization} and Bhatt--Lurie \cite{fg} (also see \cite[\S~3.2]{Mon}). We also use the author's previous work on Dieudonn\'e theory in terms of $F$-gauges \cite{Mon} (also see \cite{Mon1, madapusi}). One of the major difficulties in the proof is in working with the invariant $\mu_{\mathrm{syn}}(M,r)$ that we introduce. In \cref{stablebock}, we prove certain formal properties regarding the notion of stable Bockstein characteristic. This opens up the possibility of understanding $\mu_{\mathrm{syn}}(M,r)$ by devissage. Our proof of \cref{mainthm} proceeds via several reduction steps to ultimately reduce it to the case of $F$-gauges that are vector bundles on $(\mathrm{Spec} \,\mathbb{F}_q)^{\mathrm{syn}}$ with Hodge--Tate weights in $\left \{0, 1 \right \}$; such an $F$-gauge corresponds to $p$-divisible groups. We point out that these reduction steps are nontrivial and uses the new approach to syntomic cohomology in terms of cohomology of certain line bundles on $(\mathrm{Spec}\,\mathbb{F}_q)^{\mathrm{syn}}$, as well as the theory of isocrystals from $p$-adic Hodge theory. For a $p$-divisible group $G$, we prove certain explicit results regarding the associated $F$-gauge $M(G)$. This allows us to understand $\mu_{\mathrm{syn}} (M(G), r)$ as well as $\chi (M(G),r)$, which is then used to directly handle the case of zeta functions associated to $p$-divisible groups.

\subsection*{Notations and conventions}
\begin{enumerate}
    \item We use the language of stable $\infty$-categories \cite{luriehigher}. For an ordinary commutative ring $R,$ we will let $D(R)$ denote the derived $\infty$-category of $R$-modules, which is a stable $\infty$-category. We let $D^b(R)$ denote the full subcategory of $D(R)$ spanned by objects $K \in D(R)$ such that $H^i (K) =0$ for $|i| \gg 0$. For a quasisyntomic ring $S$, we let $(S)_\mathrm{qsyn}$ denote the quasisyntomic site of $S$ (see \cite[Variant.~4.33]{BMS2}).

    \item Let $\mathcal{C}$ be any Grothendieck site and let $\mathcal{D}$ be any presentable $\infty$-category. Then one can define the category of ``sheaves on $\mathcal{C}$ with values in $\mathcal{D}$" denoted by $\mathrm{Shv}_{\mathcal D}(\mathcal C)$ as in \cite[Def.~1.3.1.4]{spectra}. This agrees with the usual notion of sheaves when $\mathcal{D}$ is a $1$-category. Let $\mathrm{PShv}_{\mathcal{D}}(\mathcal C):= \mathrm{Fun}(\mathcal{C}^{\mathrm{op}}, \mathcal{D}).$ For an ordinary ring $R$, the classical (triangulated) derived category obtained from complexes of sheaves of $R$-modules on $\mathcal{C}$ is the homotopy category of the full subcategory of hypercomplete objects of $\mathrm{Shv}_{D(R)}(\mathcal{C})$

\item Let $G$ be a $p$-divisible group over a 
quasisyntomic ring $S.$ In this set up, for $n \ge 0,$ we can regard the group scheme of $p^n$-torsion of $G$, denoted by $G[p^n]$, as an abelian sheaf on $(S)_\mathrm{qsyn}.$ In this situation, the collection of $G[p^n]$'s naturally defines an ind-object, as well as a pro-object in $(S)_\mathrm{qsyn}.$ We will again use $G$ to denote $\mathrm{colim}\,G[p^n]$ as an object of $(S)_\mathrm{qsyn}.$ We will use $T_p(G)$ to denote $\lim G[p^n] \in (S)_\mathrm{qsyn},$ and call it the Tate module of $G.$ 

\item For a field $K$, we let $G_K$ denote the absolute Galois group of $K$. For any perfect field $k$ of characteristic $p$, we let $(\mathrm{Spec}\, k)^{\mathrm{syn}}$ denote the ``syntomification" of $\mathrm{Spec}\, k$, which is a stack \cite[Def.~4.1.1]{fg}. The derived $\infty$-category of quasicoherent sheaves on $(\mathrm{Spec}\, k)^{\mathrm{syn}}$, denoted as $D_{\mathrm{qc}}((\mathrm{Spec}\, k)^{\mathrm{syn}})$ will be called the category of $F$-gauges over $k$. We set $F\text{-}\mathrm{Gauge}(k):=D_{\mathrm{qc}}((\mathrm{Spec}\, k)^{\mathrm{syn}})$. For a non-stacky definition of the category $F\text{-}\mathrm{Gauge}(k)$, see \cite{FJ13} or \cite[\S~3.2]{Mon}. An object of $F\text{-}\mathrm{Gauge}(k)$ will simply be called an $F$-gauge. An $F$-gauge is called dualizable if it is a perfect complex when viewed as an object of $D_{\mathrm{qc}}((\mathrm{Spec}\, k)^{\mathrm{syn}}$. The full subcategory spanned by such objects will be denoted by $\mathrm{Perf}((\mathrm{Spec}\, k)^{\mathrm{syn}})$.

\item The stack $(\mathrm{Spec}\, k)^{\mathrm{syn}}$ has a canonical line bundle $\mathcal{O}\left \{1\right \}$; we call this the Breuil--Kisin twist. For an $F$-gauge $M$ and $n \in \mathbb{Z}$, we let $M \left \{ n \right \}:= M \otimes \mathcal{O}\left \{1 \right \}^{\otimes n}.$
\end{enumerate}{}

\subsection*{Acknowledgement} 
I would like to thank Niranjan Ramachandran for helpful conversations related to this paper at the annual meeting of Simons collaboration on Perfection (2024), and for an invitation to visit Maryland. I would also like to thank Kai Behrend, Bhargav Bhatt, Luc Illusie, Sujatha Ramdorai, and Xiaohan Wu for helpful conversations. During the preparation of this article, I was supported by the University of British Columbia, Vancouver.

\section{Isocrystals and syntomic cohomology}\label{someres}
Our main goal in this section is to prove certain finiteness result (\cref{corol}) involving syntomic cohomology of $F$-gauges (\cref{deffga}) that will be useful later. The notion of isocrystals play an important role in this section, and we discuss their relation with $F$-gauges. We start with some definitions.

\begin{notation}\label{nota}
 Below, we let $k$ denote any perfect field of char. $p>0$ and $K:= W(k) [\frac{1}{p}].$ Let $\varphi_K$ denote the automorphism of $K$ induced by the Witt-vector Frobenius. Let $K_{\sigma}[F]$ denote the Frobenius twisted polynomial ring over $k$, i.e., we have the relation $Fc = \sigma (c) F.$   
\end{notation}{}

\begin{definition}[$F$-gauges]\label{deffga}
    For any perfect field $k$ of characteristic $p$, we let $(\mathrm{Spec}\, k)^{\mathrm{syn}}$ denote the ``syntomification" of $\mathrm{Spec}\, k$, which is a stack as defined in \cite[Def.~4.1.1]{fg}. We set $F\text{-}\mathrm{Gauge}(k):=D_{\mathrm{qc}}((\mathrm{Spec}\, k)^{\mathrm{syn}})$, which will be called the category of $F$-gauges over $k$.
\end{definition}

\begin{remark}[Realizations]\label{notaa}
  For a smooth proper variety $X$ over $k$ one can associate a dualizable $F$-gauge $M(X)$ over $k$. By \cite{fg}, one has a map $f: X^{\mathrm{syn}} \to (\mathrm{Spec}\,k)^\mathrm{syn}$. We define $M(X):=Rf_* \mathcal{O}$. We discuss certain ``realizations" of $M(X)$ in terms of other cohomology theories. There is an open point $i: \mathrm{Spf}\,W(k) \to (\mathrm{Spec}\,k)^\mathrm{syn}$, and $i^* M(X)$ identifies with prismatic cohomology $R\Gamma_{\Prism}(X)$, which is further isomorphic to $R\Gamma_{\mathrm{crys}}(X)$ up to a Frobenius twist. There is also a closed point $j: B\mathbb{G}_m \to (\mathrm{Spec}\,k)^\mathrm{syn}$, and $j^* M(X)$ identifies with $\bigoplus_i R\Gamma (X, \Omega_X^i)[-i]$, considered as an object of $D_{\mathrm{qc}}(B\mathbb{G}_m).$ Further, there are two divisors $j_{dR,+}: \mathbb{A}^1/\mathbb{G}_m \to (\mathrm{Spec}\,k)^\mathrm{syn}$, $j_{HT,+}: \mathbb{A}^1/\mathbb{G}_m \to (\mathrm{Spec}\,k)^\mathrm{syn}$ such that, as a filtered object, $j_{dR,+}^* M(X), j_{HT,+}^* M(X)$ identifies with the Hodge filtration on de Rham cohomology $R\Gamma_{\mathrm{dR}}(X)$, and the conjugate filtration on de Rham cohomology (up to a Frobenius twist). The stack $(\mathrm{Spec}\, k)^{\mathrm{syn}}$ has a canonical line bundle $\mathcal{O}\left \{1 \right \}$ called the Breuil--Kisin twist. For any $n \in \mathbb{Z}$, by \cite[Prop.~4.4.2]{fg} (cf.~\cite[Prop.~8.4]{BMS2}), $R\Gamma((\mathrm{Spec}\, k)^{\mathrm{syn}}, M(X)\left \{n \right\})$ is isomorphic to weight $n$-syntomic cohomology $R\Gamma_{\mathrm{syn}}(X, \mathbb{Z}_p(n))$ (cf.~\cref{sad}).
\end{remark}{}

\begin{definition}[Hodge--Tate weights]
 Let $M \in F\text{-}\mathrm{Gauge}(k)$ be dualizable. Then $j^* M$ (\cref{notaa}) is a perfect complex on $B\mathbb{G}_m$, and can be identified with a finite direct sum $\bigoplus_{i \in S} M_i$, such that $M_i \in D(k)$ is nonzero for $i \in S$. The finite subset of integers $S$ will be called the Hodge--Tate weights of $M$.
\end{definition}{}

\begin{definition}[Isocrystals]\label{isocrystals}
 Let $k$ be a perfect field. An isocrystal over $k$ is a finite dimensional $K$-vector space $V$ equipped with a bijective Frobenius semilinear map $F: V \to V$. With morphisms defined in the obvious way, isocrystals form a category that we will denote by $\mathrm{Isocrys}(k).$ 
\end{definition}

\begin{example}\label{ex}
Let $r,s$ be two coprime integers with $r>0$. The $K$-vector space $$K_\sigma[F]/ K_{\sigma}[F](F^r - p^s)$$ is naturally equipped with a bijective Frobenius semilinear operator which gives it the structure of an isocrystal that will be denoted by $E_{s/r}.$ Note that the dimension of the vector space underlying $E_{s/r}$ is $r.$ Identifying the latter vector space with $K^r,$ one sees that the semilinear operator $F$ acts as follows:
$$F(x_1, \ldots, x_r) := (\varphi_K(x_2), \ldots, \varphi_K (x_{r}), p^s\varphi_K(x_1)).$$ 
\end{example}

\begin{remark}
When $s \ge 0$, the isocrystal $E_{r/s}$ admits a lattice given by $W(k)^{\oplus r}$, that is also preserved by $F$.     
\end{remark}

\begin{remark}\label{isocsyn} Let $\mathrm{Vect}((\mathrm{Spec}\, k)^\mathrm{syn})$ denote the category of vector bundles on the stack $(\mathrm{Spec}\, k)^{\mathrm{syn}}.$ In general, there is a forgetful functor $$\mathrm{Vect}((\mathrm{Spec}\, k)^\mathrm{syn}) \to \mathrm{Isocrys}(k).$$ Further,
any isocrystal can be used to define a vector bundle on $(\mathrm{Spec}\, k)^\mathrm{syn}$; this amounts to choosing a $W(k)$-lattice for the vector space underlying the isocrystal (which does not need to be preserved by $F$) (see \cite[Prop.~4.3.1]{fg}).   
\end{remark}

\cref{isocsyn} allows us to relate isocrystals and $F$-gauges. 
\begin{proposition}\label{isom}
Let $k$ be a perfect field. Then we have an equivalence of categories
$$\mathrm{Vect}((\mathrm{Spec}\, k)^\mathrm{syn}) \otimes \mathbb{Q}_p \simeq \mathrm{Isocrys}(k).$$
\end{proposition}{}

\begin{proof}
One obtains a faithful functor $\mathrm{Vect}((\mathrm{Spec}\, k)^\mathrm{syn}) \otimes \mathbb{Q}_p \to \mathrm{Isocrys}(k)$ from \cref{isocsyn} since the target category is $\mathbb{Q}_p$-linear. To see that it is full and essentially surjective, one can choose lattices.  
\end{proof}

\begin{remark}\label{isostack}Let $K$ be as in \cref{nota}.
 Let $\mathcal{I}so_K$ denote the stack defined as 
 \begin{equation}
   \mathrm{coeq} (\xymatrix{
  \mathrm{Spec}\,K\ar@<1ex>[r]^{\mathrm{id}}\ar@<-1ex>[r]_{\phi_K} & \mathrm{Spec}\,K
 })   
 \end{equation}{}
By the formal GAGA principle discussed in \cite[Lem.~3.4.11]{fg}, one can deduce a natural isomorphism 
$$\mathrm{Perf}((\mathrm{Spec}\, k)^{\mathrm{syn}})\otimes \mathbb{Q}_p \simeq \mathrm{Perf}(\mathcal{I}so_K).$$
\end{remark}

Note that $\mathrm{Vect}((\mathrm{Spec}\, k)^\mathrm{syn})$ is naturally equipped with a symmetric monoidal structure, which induces a symmetric monoidal structure on $\mathrm{Vect}((\mathrm{Spec}\, k)^\mathrm{syn}) \otimes \mathbb{Q}_p$.  As a consequence of \cref{isom}, we obtain a symmetric monoidal structure on $\mathrm{Isocrys}(k)$, which agrees with the classical one (e.g., see \cite{why}). The line bundle $\mathcal{O}\left \{1 \right \} \in \mathrm{Vect}((\mathrm{Spec}\, k)^\mathrm{syn})$ defines an object in $\mathrm{Isocrys}(k)$ that we again denote by $\mathcal{O}\left \{1 \right \},$ when no confusion is likely. 

\begin{definition}
 For $M \in \mathrm{Isocrys}(k)$ and $i \in \mathbb Z$, the $i$-th Breuil--Kisin twist of $M$ is defined by $M \otimes \mathcal{O} \left \{1 \right \}^{\otimes{i}} \in \mathrm{Isocrys}(k).$   
\end{definition}

Classically, the above twist is called the Tate twist. In this paper, we use the terminology of Breuil--Kisin twists to be compatible with the twists that appear for $F$-gauges. Concretely, this twist is described as follows. Let $M = (W, F),$ where $W$ is a $K$-vector space equipped with a bijective semilinear map $F$. Then $M \left \{ i \right \}$ is given by the pair $(W, \frac F {p^i}).$

\begin{definition}
We call an $M \in \mathrm{Isocrys}(k)$ effective if $F$ prserves some lattice. Note that for any $M \in \mathrm{Isocrys}(k),$ the isocrystal $M \left \{-i \right \}$ is effective for $i \gg 0.$   
\end{definition}

Now suppose that $k$ is algebraically closed. Then we have the following result, known as the Dieudonn\'e--Manin classification \cite{Manin}.

\begin{theorem}[Dieudonn\'e--Manin]
Let $k$ be algebraically closed. Then every isocrystal over $k$ is a direct sum of simple isocrystals. The simple isocrystals are all isomorphic to $E_{s/r}$ for coprime integers with $r>0.$     
\end{theorem}

\begin{remark}[Slopes and multiplicities]
For an isocrystal $M$ over an algebraically closed field $k$ the set of slopes of $M$ are defined to be rational numbers $s/r$ such that $E_{s/r}$ appears in the direct sum decomposition of $M$. The $K$-dimension of the direct sum of all summands of $M$ that are isomorphic to $E_{s/r}$ will be called the multiplicity of the slope $s/r$. 
\end{remark}{}

\begin{example}
 The isocrystal $E_{s/r}$ only has $s/r$ as a slope, which is of multiplicity $r$.    
\end{example}{}

\begin{remark}
If $k$ is algebraically closed and $M \in \mathrm{Isocrys}(k),$ then $M$ is effective if and only if in the direct sum decomposition of $M$ into simple objects, the $E_{s/r}$'s that appear all have $s \ge 0.$ Equivalently, all slopes of $M$ must be nonnegative. 
\end{remark}{}
\begin{proposition}\label{imp}
Let $k$ be algebraically closed and $M \in \mathrm{Isocrys}(k).$ Then cokernel of the map $$F - p^i: M \to M $$ is trivial and the kernel is a finite dimensional $\mathbb{Q}_p$-vector space for any $i \in \mathbb{Z}.$  
\end{proposition}

\begin{proof}
 By applying a suitable Breuil--Kisin twist, one can assume that $M$ is effective and $i \ge 0.$ By the Dieudonn\'e--Manin classification, one can can further assume that $M = E_{s/r}$ for $s \ge 0, r>0$.

First, we give an argument for the kernel. If $0 \neq x \in E_{s/r}$, then $F(x) = p^i x$ implies that $F^r (x) = p^{ir}(x).$ Using the concrete description from \cref{ex} and writing $x= (x_1, \ldots, x_r),$ we see that one must have $p^s \varphi_K^r(x_i) = p^{ir}x_i$ for all $1 \le i \le r.$ Since $x \ne 0,$ there exists a $j$ such that $x_j \ne 0.$ Since $\varphi_K$ preserves $p$-adic valuations, we must have $s = ir.$ Since $s,r$ are coprime, this forces $r=1$ and $s=i.$ In that case, $F(x) = p^ix$ amounts to $x$ being a fixed point of $\varphi_K,$ which happens if and only if $x \in \mathbb{Q}_p.$ This proves the finite dimensionality of the kernel as a $\mathbb{Q}_p$-vector space.

To prove that the cokernel is zero, we again use the concrete description from \cref{ex}. Using that, we are required to prove that for any $y_i \in K$, there exists $x_i \in K$ such that 
$$(\varphi_K(x_2), \ldots, \varphi_K (x_{r}), p^s\varphi_K(x_1)) - (x_1, \ldots, x_r) = (y_1, \ldots, y_r).$$
Since $\varphi_K$ is an isomorphism, it follows that this is equivalent to solving the equations 

$$x_1 = p^s \varphi_K^r(x_1) - \varphi_K^{r-1}(y_r) - \ldots - \varphi_K(y_2) - y_1$$ and $$\varphi_K(x_{t+1}) - x_t = y_t$$ for $1 \le t \le r-1.$
Thus, it is enough to prove that the map $1 - p^s \varphi^r: W(k) \to W(k)$ is surjective. Since $W(k)$ is $p$-adically complete, it is enough to prove surjectivity modulo $p.$ Thus, for $s >0$, the surjectivity is immediate. For $s =0,$ the surjectivity follows since $k$ is algebraically closed. This finishes the proof.   
\end{proof}{}

\begin{definition}\label{verysad}
     Suppose that $k$ is any perfect field as in \cref{nota}. Let $M$ be a dualizable $F$-gauge over $k$. There is a natural map $\mathrm{Spf}\, W(k) \to (\mathrm{Spec}\, k)^{\mathrm{syn}}$ (see \cite[Rmk.~4.1.2]{fg}). The pullback of $M$ along this map is a perfect complex of $W(k)$-modules that will be denoted as $M^u$; this can be regarded as the underlying module of the $F$-gauge $M$. In this scenario, $M^u$ is naturally equipped with a Frobenius semilinear endomorphism. Further, $M_K:=M^u \otimes_{W(k)} K$ is naturally a dualizable object of $D(K).$ See \cite[\S~3.2]{Mon} for more details.
\end{definition}

\begin{definition}[Syntomic cohomology of $F$-gauges]\label{sad}
 For an $F$-gauge $M$ over $k$, we define $$R\Gamma_{\mathrm{syn}}(M, \mathbb{Z}_p(n)):= R\Gamma ((\mathrm{Spec}\, k)^{\mathrm{syn}}, M \left \{ n \right \}).$$ We refer to $R\Gamma_{\mathrm{syn}}(M, \mathbb{Z}_p(n)) $ as the weight $n$ syntomic cohomology of $M$. We define $R\Gamma_{\mathrm{syn}}(M, \mathbb{Q}_p(n)) := R\Gamma_{\mathrm{syn}}(M, \mathbb{Z}_p(n))\otimes_{\mathbb{Z}_p} \mathbb{Q}_p$.   
\end{definition}{}

\begin{proposition}\label{co1}
Let $M$ be a perfect complex in $(\mathrm{Spec}\, k)^\mathrm{syn},$ where $k$ is algebraically closed. Then $H^i_{\mathrm{syn}}(M, \mathbb{Q}_p(n)) \simeq (H^i (M_K))^{F = p^n}$ for $n \in \mathbb{Z}.$ 
\end{proposition}{}

\begin{proof}
 Let $M_K$ denote the object of $D(K)$ associated to $M$ by pullback, which is a perfect complex. $M_K$ is naturally equipped with a Frobenius $F_{M_K}$, that induces a Frobenius $F$ on $H^i (M_K)$, which is a bijection. Since $H^i (M_K)$ is finite dimensional as a $K$-vector space, it is naturally an isocrytstal over $k$.

Note that we have a natural functor $$\mathrm{Perf}((\mathrm{Spec}\, k)^\mathrm{syn}) \to \mathrm{Perf}((\mathrm{Spec}\, k)^\mathrm{syn}) \otimes{\mathbb{Q}_p},$$ where the target is naturally a symmetric monoidal $\infty$-category. Concretely, $\mathrm{Perf}((\mathrm{Spec}\, k)^\mathrm{syn}) \otimes{\mathbb{Q}_p}$ is the category of perfect complexes on the stack $\mathcal{I}so_K$ defined in \cref{isostack}. By construction, we have 
$R\Gamma ((\mathrm{Spec}\,k)^\mathrm{syn}, M \left \{n\right \}) \otimes \mathbb{Q}_p \simeq R\Gamma(\mathcal{I}so_K, M \left \{n\right \}).$
Further, $M_K$ identifies with the pullback of $M$ along the map $\mathrm{Spec}\, K \to \mathcal{I}so_K.$ It follows that $$R\Gamma (\mathcal{I}so_K, M \left \{n\right \}) \simeq \mathrm{Fib} (M_K \xrightarrow{F_{M_K} - p^n}  M_K).$$

This gives a fiber sequence
$$R\Gamma_{\mathrm{syn}}(M, \mathbb{Q}_p(n)) \xrightarrow{\textcolor{white}{somethi}} M_K \xrightarrow{F_{M_K} - p^n}  M_K. $$ The claim now follows from the associated long exact sequence and applying \cref{imp}. 
\end{proof}
\begin{corollary}\label{corol}
Let $M$ be a perfect complex in $(\mathrm{Spec}\, k)^\mathrm{syn},$ where $k$ is algebraically closed. Then $H^i_{\mathrm{syn}}(M, \mathbb{Q}_p(n))$ is a finite dimensional $\mathbb{Q}_p$-vector space for all $i \in \mathbb{Z}.$     
\end{corollary}{}
\begin{proof}
Follows from \cref{imp} and \cref{co1}.    
\end{proof}

\begin{remark}\label{examplebad}
For an $F$-gauge $M$ over an algebraically closed field $k$, the $\mathbb{Z}_p$-module $H^i_{\mathrm{syn}}(M, \mathbb{Z}_p(n))$ is not in general finitely generated. For a supersingular elliptic curve $E$ over an algebraically closed field $k$, the $\mathbb{Z}_p$-module $H^3 _{\mathrm{syn}}(E \times E, \mathbb{Z}_p(1))$ is not finitely generated. As \cref{corol} shows, this issue disappears after inverting $p$.
\end{remark}

\section{Zeta function of $F$-gauges}
In this section, we work with $F$-gauges over finite fields. We start by defining the zeta function associated to such $F$-gauges. We also discuss a descent spectral sequence that relates Galois cohomology and syntomic cohomology (\cref{sad}), and use it to prove a result about syntomic cohomology (\cref{dim0}) and order of vanishing of zeta functions (\cref{orderofvan}). 

\begin{notation}\label{two}
 Let $p$ be a fixed prime. Let $k$ be the finite field $\mathbb{F}_q$, where $q = p^r$. Let $K= W(k)[\frac  1 p]$. For a dualizable $F$-gauge over $k$, the underlying $W(k)$-module $M^u$ is naturally a perfect complex of $W(k)$-modules equipped with a Frobenius. The finite dimension $K$-vector space $H^i (M^u)[\frac 1 p]$ identifies with $H^i (M_K),$ on which the induced Frobenius action $F$ is bijective (see \cref{verysad}). This equips $H^i (M_K)$ with the structure of an isocrystal over $k$ (\cref{isocrystals}).   
\end{notation}

\begin{definition}[Zeta function]\label{zetafu}
 In the set up of \cref{two}, the zeta function of $M$, denoted as $Z(M, t)$, is be defined as follows

$$Z(M,t):= \prod_{i \ge 0} \det (1- t F^r| H^i(M_K))^{(-1)^{i+1}}.$$We set $\zeta (M,s) := Z(M, q^{-s}).$
\end{definition}

Before studying the zeta function further, we first discuss the construction of the descent spectral sequence. Let $M \in \FGauge (k)$ be dualizable. Let $\overline{k}$ be an algebraic closure of $k.$ Let $\mathrm{Pro\acute{e}t}(k)$ denote the pro-\'etale site \cite{proetale} of $k.$ For any $\Spec\,A \in \mathrm{Pro\acute{e}t}(k),$ there is a map $u: (\mathrm{Spec}\, A)^\mathrm{syn} \to (\mathrm{Spec} \, k)^\mathrm{syn}$. The association $A \mapsto R\Gamma((\mathrm{Spec}\, A)^\mathrm{syn}, u^* M \left \{n\right \})$ defines a $D(\mathbb{Z}_p)$-valued sheaf on $\mathrm{Pro\acute{e}t}(k).$ This sheaf, viewed as a functor, $\mathrm{Pro\acute{e}t}(k)^\mathrm{op} \to D(\mathbb{Z}_p)$ will be denoted by $\mathcal{M}\left \{n\right \}.$

Let $\mathrm{Shv}_{D(\mathbb{Z}_p)}(\mathrm{Pro\acute{e}t}(k))$ denote the category of $D(\mathbb{Z}_p)$-valued sheaves on $\mathrm{Pro\acute{e}t}(k).$ By definition, any sheaf $\mathrm{Pro\acute{e}t}(k)^\mathrm{op} \to D(\mathbb{Z}_p)$ preserves limits. For any $\mathcal{F} \in \mathrm{Shv}_{D(\mathbb{Z}_p)}(\mathrm{Pro\acute{e}t}(k)),$ there is a map 
$$ \gamma : \mathcal{F}(\overline{k}) \to \mathcal{F}(\overline{k})$$ induced by the $\sigma^r$, where $\sigma$ is the Frobenius on $\overline{k}.$ Since $k \simeq \mathrm{Fib}(k \xrightarrow{\sigma^r - 1} k)$, by sheafiness, it follows that 
\begin{equation}\label{c}
   \mathcal{F}(k) \simeq \mathrm{Fib}(\mathcal{F}(\overline{k}) \xrightarrow{\gamma - 1} \mathcal{F}(\overline{k}))). 
\end{equation}{}

In terms of continuous cohomology of $G_k$, it can be restated as 
$$R\Gamma_{\mathrm{cont}} (G_k, \mathcal{F}(\overline{k})) \simeq \mathcal{F}(k). $$
Applying this to $\mathcal{M}\left \{n \right \}$, and denoting $u^* M$ by $\overline{M}$, we obtain
$$R\Gamma_{\mathrm{cont}}(G_k, R\Gamma_{\mathrm{syn}}(\overline{M}, \mathbb{Z}_p(n))) \simeq R\Gamma_{\mathrm{syn}}({M}, \mathbb{Z}_p(n)).$$
This gives the desired descent spectral sequence 
\begin{equation}\label{deu}
  E_2^{ij}:= H^i _{\mathrm{cont}}(G_k, H^j_{\mathrm{syn}}(\overline{M}, \mathbb{Z}_p(n))) \implies H^{i+j}_{\mathrm{syn}}(M, \mathbb{Z}_p(n)).   
\end{equation}{}

By \cref{c}, the continuous cohomological dimension of $G_k$ is $1.$ Therefore, the above spectral sequence degenerates and we obtain short exact sequences
\begin{equation}\label{shortexact}
  0 \to H^1 (G_k, H^{j-1}_{\mathrm{syn}}(\overline{M}, \mathbb{Z}_p(n))) \to H^{j}_{\mathrm{syn}}(M, \mathbb{Z}_p(n)) \to H^0 (G_k, H^j_{\mathrm{syn}}(\overline{M}, \mathbb{Z}_p(n))) \to 0 .  
\end{equation}{}

\begin{remark}
Since $M \in \FGauge(k)$ is dualizable and $k$ is a finite field, $H^j_{\mathrm{syn}} (M, \mathbb{Z}_p(n))$ is a finitely generated $\mathbb{Z}_p$-module for all $j$ (cf.~\cite[Prop.~4.5.1]{fg}). As a consequence of \cref{shortexact}, $H^i (G_k, H^q_{\mathrm{syn}}(\overline{M}, \mathbb{Z}_p(n)))$ is also finitely generated as $\mathbb{Z}_p$-modules for all $i.$ But $H^q_{\mathrm{syn}}(\overline{M}, \mathbb{Z}_p(n))$ itself is \emph{not} in general finitely generated (see \cref{examplebad}). Nevertheless, by \cref{corol}, $H^q_{\mathrm{syn}}(\overline{M}, \mathbb{Z}_p(n))[1/p]$ is a finitely generated $\mathbb{Q}_p$-vector space.   
\end{remark}{}

Thus, we have the following lemma.

\begin{lemma}\label{h}
$\mathrm{rank}\, H^0 (G_k, H^q_{\mathrm{syn}}(\overline{M},\mathbb{Z}_p(n))) = \mathrm{rank}\, H^1 (G_k, H^q_{\mathrm{syn}}(\overline{M},\mathbb{Z}_p(n))).$    
\end{lemma}

\begin{proof}
 We have an induced map $H^q_{\mathrm{syn}}(\overline{M}, \mathbb{Z}_p(n))[1/p] \xrightarrow{\gamma -1} H^q_{\mathrm{syn}}(\overline{M}, \mathbb{Z}_p(n))[1/p]$ of finite dimensional $\mathbb{Q}_p$-vector spaces (\cref{corol}). Its kernel and cokernel are respectively given by $H^0 (G_k, H^q_{\mathrm{syn}}(\overline{M},\mathbb{Z}_p(n)))[1/p]$ and $H^1 (G_k, H^q_{\mathrm{syn}}(\overline{M},\mathbb{Z}_p(n)))[1/p]$, and therefore must have the same dimension as $\mathbb{Q}_p$-vector spaces. 
\end{proof}

\begin{proposition}\label{dim0}
Let $M$ be a dualizable $F$-gauge over a finite field $k$. Then $$\sum_{i \ge 0} (-1)^i \mathrm{rank}\, H^i_{\mathrm{syn}} (M, \mathbb{Z}_p(n))=0.$$   
\end{proposition}

\begin{proof}
 Follows from the above lemma and the short exact sequence \cref{shortexact}. \end{proof}

\begin{remark}
Combining a result of Ekedahl \cite[Thm.~5.3]{eke} with \cref{dim0}, one obtains a different proof of \cite[Thm.~0.1(a)]{MRp}. In fact, we obtain a generalization of \cite[Thm.~0.1(a)]{MRp} that does not require the semisimplicity assumption.    
\end{remark}

Now we proceed towards proving \cref{orderofvan}, which gives an expression of order of vanishing of the zeta function associated to $F$-gauges (\cref{zetafu}) in terms of syntomic cohomology. We will begin by fixing some notations and record a lemma.

As before, let $q = p^r$, and consider the finite field $\mathbb{F}_q.$ Let $M \in \mathrm{Isocrys}(\mathbb{F}_q)$. We set $K:= W(\mathbb{F}_q)[1/p]$, and $\overline{K}:= W(\overline{\mathbb F}_q)[1/p].$ We write $M= (V, F),$ where $F$ denotes the semilinear endomorphism of $V$. In this case, $(\overline{V}, \overline{F}):=(V \otimes_K \overline{K}, F \otimes \varphi_{\overline{K}})$ determines an isocrystal $\overline{M}$ over $\overline{\mathbb F}_q$. Note that since $\varphi_K^r = \mathrm{id},$ $F^r : V \to V$ is $K$-linear. Further, $\mathrm{id} \otimes \varphi_{\overline{K}}^r: \overline{V} \to \overline{V}$ is a $K$-linear map, which we denote by $\gamma.$ By \cref{imp}, the $\mathbb{Q}_p$-vector space $\overline{V}^{\overline{F}= p^n}$ is finite dimensional. Since $\overline{F}$ and $\gamma$ commutes, we have an induced $\mathbb{Q}_p$-linear map $\gamma': \overline{V}^{\overline{F}= p^n} \to \overline{V}^{\overline{F}= p^n}.$ In this set up, we have the following lemma.

\begin{lemma}\label{co2}
For any $n \in \mathbb{Z}$, we have an isomorphism $$ V^{F^r=q^n} \simeq ((\overline{V}^{\overline{F}= p^n})^{\gamma' = 1}) \otimes_{\mathbb{Q}_p} K$$ of $K$-vector spaces.   
\end{lemma}{}

\begin{proof}
Since $\overline{F}$ and $\gamma$ commutes, we have $V^{F^r = q^n} \simeq (\overline{V}^{\overline{F}^r= q^n})^{\gamma = 1}$ as $K$-vector spaces. By the Dieudonn\'e--Manin classification, we can write the isocrystal as a finite direct sum$$ \overline{V} \simeq \bigoplus W_j$$ where $W_j$ is a direct sum of copies $E_{u/v}$ such that $u/v =j.$ Using the concrete description of $E_{u/v}$ from \cref{ex} and $p$-adic valuation considerations, one sees that $ W_j^{\overline{F}^r= q^n}=0$ if $j \ne n.$ Similarly, $W_j^{\overline{F}= p^n}=0$ for $j \ne n$. Suppose that $W_n \simeq (E_{n/1})^{\oplus t}$.
Then by a direct calculation we have $W_n^{\overline{F}^r= q^n} \simeq K^{\oplus t}$. The Galois action restricts to a $K$-linear map $\gamma: K^{\oplus t} \to K^{\oplus t}.$ Also, one has $W_n^{\overline{F}= p^n} \simeq \mathbb{Q}_p^{\oplus t}$ and the Galois action restricts to $\mathbb{Q}_p$-linear map $\gamma': \mathbb{Q}_p^{\oplus t} \to \mathbb{Q}_p^{\oplus t}. $
Therefore, it follows that $\gamma$ identifies with $\gamma' \otimes_{\mathbb{Q}_p} K$. Since $(\cdot) \otimes_{\mathbb{Q}_p} K$ is exact, we obtain 
$$(\overline{V}^{\overline{F}^r= q^n})^{\gamma = 1} \simeq ({W_n}^{\overline{F}^r= q^n})^{\gamma = 1} \simeq (K^{\oplus t})^{\gamma =1} \simeq (\mathbb{Q}_p^{\oplus t})^{\gamma'= 1}\otimes_{\mathbb{Q}_p}K \simeq (W_n^{\overline{F}= p^n})^{\gamma' =1} \otimes_{\mathbb{Q}_p} K ,$$ the latter is further isomorphic to $(\overline{V}^{\overline{F}= p^n})^{\gamma'=1}\otimes_{\mathbb{Q}_p}K.$ This finishes the proof.
\end{proof}

\begin{proposition}\label{hungry}
Let $M$ be a dualizable $F$-gauge over $\mathbb{F}_q$, where $q=p^r$. Then $$(H^i (M_K)^{F^r = q^n} \simeq H^i_{\mathrm{syn}}(\overline{M}, \mathbb{Q}_p(n))^{G_{\mathbb{F}_q}} \otimes_{\mathbb{Q}_p}K. $$ Here $K:= W(\mathbb{F}_q)[1/p]$ and $\overline{{M}} \in \FGauge (\overline{\mathbb{F}}_q)$ is the base change of $M$.
\end{proposition}{}
\begin{proof}
Follows from \cref{co1} and \cref{co2}.   
\end{proof}{}

\begin{proposition}\label{orderofvan}
Let $M$ be a dualizable $F$-gauge over $\mathbb{F}_q$. Let us assume that $q^n$ is a semisimple eigenvalue of $F^r$. Then for any $n \in \mathbb{Z}$, the order of zero of $\zeta(M, s)$ at $s=n$ is given by 
$$\mathrm{ord}_{s=n} \zeta({M}, s)=\sum_{i \ge 0} (-1)^i i \cdot \mathrm{rank}\,H^i _{\mathrm{syn}} (M, \mathbb{Z}_p(n)).$$    
\end{proposition}

\begin{proof}
 Let $P_i (t):= \mathrm{det}(1- t F^r| H^i(M_K)).$ Then the order of zero of $P_i$ at $q^{-n}$ is the same as the order of zero of $q^n$ of the characteristic polynomial of $F^r$ on $H^i( M_K)$. By the semisimplicity assumption, this is the same as $\dim_K(H^i (M_K))^{F^r = q^n}$. The latter is equal to $\mathrm{rank} H^i_{\mathrm{syn}} (\overline{M}, \mathbb{Z}_p(n))^{G_{\mathbb{F}_q}}$ by \cref{hungry}. Let us denote $u_i:= \mathrm{rank}_{\mathrm{syn}}H^i (\overline{M}, \mathbb{Z}_p(n))^{G_{\mathbb{F}_q}}.$ By \cref{shortexact} and \cref{h}, it follows that 
$$ \mathrm{rank} H^i_{\mathrm{syn}}(M, \mathbb{Z}_p(n))= u_{i-1}+ u_i.$$
Thus $\sum_{i \ge 0} (-1)^i i \cdot \mathrm{rank} H^i _{\mathrm{syn}} (M, \mathbb{Z}_p(n)) = \sum_{i \ge 0} (-1)^{i+1} u_i$; the latter is clearly equal to the order of the zero of $\zeta(M, s)$ at $s=n$.
\end{proof}

\begin{remark}
 Combining \cite[Thm.~5.3]{eke} with \cref{dim0}, one recovers \cite[Thm.~0.1(b)]{MRp}. Although we impose certain semisimplicity assumption in \cref{orderofvan}, our main result (\cref{mainthm}) regarding special values will have no such assumptions. In order to achieve that, we need some new notions that we will develop in the subsequent sections. 
\end{remark}{}

\section{Bockstein complexes}\label{bockst}
In this section, we construct a generalization of the Bockstein map, as well as the Bockstein spectral sequence by introducing certain Bockstein complexes. We study general properties of Euler characteristic (when it exists) of certain Bockstein complexes. The constructions and the results from this section are used in \cref{stablebock} to define the stable Bockstein characteristic, which will be used later in the context of special values of zeta functions.

\begin{construction}\label{bock}
 Let $M \in D^b (\mathbb{Z}_p).$ Let $\theta: M \to M$ be a map. Let $R:= \mathrm{fib}(\theta).$ Note that there is a natural composite map $$\mathrm{fib}(\theta) \to M \xrightarrow{1} M \to \mathrm{cofib}(\theta).$$ Since $\mathrm{fib}(\theta)[1] \simeq \mathrm{cofib}(\theta)$, we obtain a map $\beta: R \to R[1].$ Taking cohomology, we obtain a complex (where $H^0 (R)$ is in degree $0$) 
$$\mathrm{Bock}^\bullet(M, \theta):=\ldots \to H^{i-1}(R) \to H^i (R ) \to H^{i+1} (R) \to \ldots .$$
\end{construction}

\begin{definition}
 For a $P \in D^b (\mathbb{Z}_p)$ such that $H^i (P)$ is a finite length $\mathbb{Z}_p$-module for all, we define the length Euler characteristic of $P$ to be $\sum (-1)^i \mathrm{length}_{\mathbb{Z}_p} H^i (P)$. We denote it by $\chi^l (P).$   
\end{definition}{}

\begin{definition}[Bockstein characteristic]

The Bockstein characteristic of $(M, \theta)$ is defined to be $\chi^l (\mathrm{Bock}^\bullet(M, \theta)),$ when it exists. In this situation, we will also simply denote it by $\chi^{l, \mathrm{B}}(R),$ where $R:= \mathrm{fib}(\theta).$  
\end{definition}

\begin{remark}
 Note that in the set up of the above definition, if $\chi^l (R)$ is defined (where $R = \mathrm{fib}(\theta)$), then it is the same as the Bockstein characteristic of $(M, \theta).$ However, a crucial point here is that $\chi^{l,B} (R)$ can be defined even if $\chi^l (R)$ is not defined.   
\end{remark}

We now explain how to construct the Bockstein spectral sequence in this generality; the complex $\mathrm{Bock}^\bullet (M, \theta)$ can be seen in the $E_1$-page of this spectral sequence.

\begin{construction}
 Let $\widehat{M}:= R\varprojlim \mathrm{cofib}(M \xrightarrow{\theta^n}M)$. Then $(\ldots \widehat{M} \xrightarrow{\theta} \widehat{M} \xrightarrow{\theta} \widehat{M} \xrightarrow{\theta} \ldots)$ defines a (complete) decreasing $\mathbb{Z}$-indexed filtration on $\widehat{M}$ whose graded pieces are all isomorphic to $\mathrm{cofib}(\theta) \simeq R[1].$ Up to changing the indexing and using [Lurie,\S 11 ], one obtains a spectral sequence whose $E_1$-page identifies with $\mathrm{Bock}^\bullet (M, \theta).$
\end{construction}
\begin{construction}\label{breka}
By viewing $M^{\theta,*}:=(\ldots \widehat{M} \xrightarrow{\theta} \widehat{M} \xrightarrow{\theta} \widehat{M} \xrightarrow{\theta} \ldots)$ as a decreasing $\mathbb{Z}$-indexed filtration, one can also view it as an object of the \emph{filtered derived category}. Note that the filtered derived category is equipped with the Beilinson $t$-structure, and one can take cohomology objects $H^n_{\mathrm{B}} (\cdot)$ with respect to this $t$-structure, which are certain cochain complexes. By construction, it follows that 
$H^0_{\mathrm{B}}(M^{\theta,*}) \simeq \mathrm{Bock}^{\bullet}(M, \theta)[1]$ as complexes. More generally we have 
$$H^n_{\mathrm{B}}(M^{\theta,*}) \simeq \mathrm{Bock}^{\bullet}(M, \theta)[n+1].$$
\end{construction}

\begin{example}\label{e}
 Let $M$ be a discrete $\mathbb{Z}_p$-module and let $\theta: M \to M$ be an endomorphism. We will use $M^\theta$ to denote $\mathrm{Ker}(\theta)$ and $M_\theta$ to denote $\mathrm{Coker}(\theta)$. Then $\mathrm{Bock}^\bullet (M, \theta)$ concretely identifies with the two term complex $[M^\theta \to M_{\theta}]$, where $M^\theta$ is in degree $0.$ The map here is given as the composite $M^\theta \to M \to M_\theta$ of the natural maps. The Bockstein characteristic of $(M,\theta)$ exists if and only if the map $M^\theta \to M_{\theta}$ has kernel and cokernel of finite length. 
\end{example}{}

\begin{remark}\label{semisimple}
In \cref{e}, suppose that $M^\theta$ and $M_\theta$ are finitely generated. Then the Bockstein characteristic of $(M,\theta)$ exists if and only if $M^\theta [\frac{1}{p}] \to M_{\theta}[\frac{1}{p}]$ is an isomorphism. Further, if we additionally assume that $M[\frac 1 p]$ is a finite dimensional $\mathbb{Q}_p$-vector space, then it follows that the Bockstein characteristic of $(M,\theta)$ exists if and only if $0$ is a semisimple eigenvalue of $\theta [\frac 1 p].$  
\end{remark}{}

\begin{remark}\label{to}
 Suppose that $M \in D^b(\mathbb{Z}_p)$ and we are in the set up of \cref{bock}. By abuse of notation, we will denote the endomorphism induced on $H^i (M) \to H^i (M)$ by $\theta$ again. This allows us to form the Bockstein complex $\mathrm{Bock}^\bullet (H^i (M), \theta)$. This complex fits in a diagram of the form
\begin{center}
\begin{tikzcd}
0           & H^{i-1}(M)^\theta \arrow[d] \arrow[l] & H^{i-1}(R) \arrow[d] \arrow[l] &                                   &             \\
0 \arrow[r] & H^{i-1}(M)_\theta \arrow[r]           & H^{i}(R) \arrow[d] \arrow[r]   & H^i(M)^\theta \arrow[d] \arrow[r] & 0           \\
            &                                       & H^{i+1}(R)                     & H^i(M)_\theta \arrow[l]           & 0 \arrow[l]
\end{tikzcd}  
\end{center}{}
 By a diagram chase, we obtain a short exact sequence that calculates cohomology of $\mathrm{Bock}^\bullet(M, \theta)$ as
$$0 \to H^1 (\mathrm{Bock}^\bullet(H^{i-1}(M), \theta)) \to H^i (\mathrm{Bock}^\bullet(M, \theta)) \to H^0 (\mathrm{Bock}^\bullet(H^i(M), \theta)) \to 0.$$
\end{remark}{}

\begin{corollary}\label{alternating}
Suppose that $M \in D^b(\mathbb{Z}_p)$ and we are in the set up of \cref{bock}. By \cref{to}, it follows that the Bockstein characteristic of $(M, \theta)$ exists if and only if the Bockstein characteristic of $(H^i (M), \theta)$ exists for all $i.$ In such a situation, it follows that 
$$\chi^l (\mathrm{Bock}^\bullet(M, \theta)) = \sum_{i} (-1)^i \chi^l (\mathrm{Bock}^\bullet(H^i(M), \theta)).$$
\end{corollary}{}

We must point out that the existence of the Bockstein characteristic is a subtle matter (see \cref{semisimple}). In particular, it is \emph{not} stable under taking extensions. For example, let us look at the following diagram
\begin{center}
\begin{tikzcd}
0 \arrow[r] & \mathbb{Z}_p \arrow[d, "\theta_1"] \arrow[r] & \mathbb{Z}_p \oplus \mathbb{Z}_p \arrow[r] \arrow[d, "\theta_2"] & \mathbb{Z}_p \arrow[d, "\theta_3"] \arrow[r] & 0 \\
0 \arrow[r] & \mathbb{Z}_p \arrow[r]                       & \mathbb{Z}_p \oplus \mathbb{Z}_p \arrow[r]                       & \mathbb{Z}_p \arrow[r]                       & 0
\end{tikzcd} 
\end{center}{}
where $\theta_1, \theta_3$ are the zero map and $\theta_2$ is defined by $(c,d)\mapsto (d,0)$. Even though the Bockstein characteristics of $(\mathbb{Z}_p, \theta_1), (\mathbb{Z}_p, \theta_3)$ exist, the Bockstein characteristic of $(\mathbb{Z}_p \oplus \mathbb{Z}_p, \theta_2)$ does not exist.

The above example also shows that the maps $$\mathrm{Bock}^\bullet (\mathbb{Z}_p, \theta_1) \to \mathrm{Bock}^\bullet (\mathbb{Z}_p \oplus \mathbb{Z}_p, \theta_2) \to \mathrm{Bock}^\bullet (\mathbb{Z}_p, \theta_3)$$ do not form a fiber sequence.

\begin{lemma}\label{lemm}
Let $M',M, M''$ be discrete $\mathbb{Z}_p$-modules. Suppose that we have a diagram  
\begin{center}
\begin{tikzcd}
0 \arrow[r] & M' \arrow[r] \arrow[d, "\theta_1"] & M \arrow[r] \arrow[d, "\theta_2"] & M'' \arrow[r] \arrow[d, "\theta_3"] & 0 \\
0 \arrow[r] & M' \arrow[r]                       & M \arrow[r]                       & M'' \arrow[r]                       & 0
\end{tikzcd}   
\end{center}where the rows are exact. Assume that $M'^{\theta_1}, M''^{\theta_3}, M'_{\theta_1}, M''_{\theta_3}$ are finitely generated $\mathbb{Z}_p$-modules, and the Bockstein characteristics of $(M', \theta_1), (M, \theta_2), (M'', \theta_3)$ exist. Then 
$$\chi^l (\mathrm{Bock}^\bullet(M, \theta_2))= \chi^l (\mathrm{Bock}^\bullet(M', \theta_1)) + \chi^l (\mathrm{Bock}^\bullet(M'', \theta_3)).$$
\end{lemma}{}

\begin{proof}
Note that by the snake lemma, there is a map $s: M''^{\theta_3} \to M'_{\theta_1} $. We also have a diagram (with exact rows)
\begin{center}
\begin{tikzcd}
0 \arrow[r] & M'^{\theta_1} \arrow[d] \arrow[r] & M^{\theta_2} \arrow[r] \arrow[d] & M''^{\theta_3} \arrow[d] &   \\
            & M'_{\theta_1} \arrow[r]           & M_{\theta_2} \arrow[r]           & M''_{\theta_3} \arrow[r] & 0
\end{tikzcd}  
\end{center}{}
such that after inverting $p$, by our hypothesis, the vertical maps induce isomorphisms. This implies that $s[\frac{1}{p}]$ is the zero map. Therefore, $\mathrm{Im}(s)$ is a finite length $\mathbb{Z}_p$-module.

 By the snake lemma, we obtain a diagram (with exact rows)
\begin{center}
\begin{tikzcd}
0 \arrow[r] & M'^{\theta_1} \arrow[d] \arrow[r] & M^{\theta_2} \arrow[r] \arrow[d] & \mathrm{Ker}(s) \arrow[d] \arrow[r] & 0 \\
0 \arrow[r] & \mathrm{Coker}(s) \arrow[r]       & M_{\theta_2} \arrow[r]           & M''_{\theta_3} \arrow[r]            & 0
\end{tikzcd}
\end{center}
It follows that the kernel and cokernel of the maps $\mathrm{Ker}(s) \to M''_{\theta_3}$ and $M'^{\theta_1} \to \mathrm{Coker}(s)$ are both finite length. We may view them as two term complexes with the source of the maps being in degree $0$ and denote them by $S_1, S_2$ respectively. Thus,
$$\chi^{l}(\mathrm{Bock}^\bullet(M, \theta_2)) = \chi^l(S_1) + \chi^l(S_2).$$
Finally, we have the following two diagrams (with exact rows)
\begin{center}
    \begin{tikzcd}
0 \arrow[r] & \mathrm{Ker}(s) \arrow[r] \arrow[d]     & M''^{\theta_3} \arrow[r] \arrow[d] & \mathrm{Im}(s) \arrow[r] \arrow[d] & 0 \\
0 \arrow[r] & M''_{\theta_3} \arrow[r, "\mathrm{id}"] & M''_{\theta_3} \arrow[r]           & 0                                  &  
\end{tikzcd}

\begin{tikzcd}
            & 0 \arrow[r] \arrow[d]    & M'^{\theta_1} \arrow[r] \arrow[d] & M'^{\theta_1} \arrow[r] \arrow[d] & 0 \\
0 \arrow[r] & \mathrm{Im}(s) \arrow[r] & M'_{\theta_1} \arrow[r]           & \mathrm{Coker}(s) \arrow[r]       & 0
\end{tikzcd}
\end{center}which implies that $\chi^l(\mathrm{Bock}^\bullet (M'', \theta_3)) = \chi^l(S_1) + \mathrm{length}(\mathrm{Im}(s))$ and $\chi^l(\mathrm{Bock}^\bullet (M', \theta_1)) = \chi^l(S_2) - \mathrm{length}(\mathrm{Im}(s))$. Combining the three equalities together gives the desired result.
\end{proof}{}

\begin{proposition}\label{be}
Let $(M', \theta_1), (M, \theta_2)$ and $(M'', \theta_3)$ be such that we have a diagram
\begin{center}
\begin{tikzcd}
M' \arrow[r] \arrow[d, "\theta_1"'] & M \arrow[r] \arrow[d, "\theta_2"'] & M'' \arrow[d, "\theta_3"'] \\
M' \arrow[r]                       & M \arrow[r]                      & M''                       
\end{tikzcd}
\end{center}{}
 where $M'\to M \to M''$ is a fiber sequence in $D^b(\mathbb{Z}_p).$ Further, assume that $\mathrm{fib}(\theta_1), \mathrm{fib}(\theta_3)$ are perfect complexes and $\chi^l(\mathrm{Bock}^{\bullet}(M', \theta_1)), \chi^l(\mathrm{Bock}^{\bullet}(M, \theta_2)), \chi^l(\mathrm{Bock}^{\bullet}(M'', \theta_3))$ exist. Then we have
$$\chi^l(\mathrm{Bock}^{\bullet}(M, \theta_2)) = \chi^l(\mathrm{Bock}^{\bullet}(M', \theta_1)) + \chi^l(\mathrm{Bock}^{\bullet}(M'', \theta_3)). $$

\end{proposition}

\begin{proof}By hypothesis, it follows that $\mathrm{fib}(\theta_3)$ is also a perfect complex. By the long exact sequence from \cref{to}, we see that $H^i(M')^{\theta_1}, H^i(M')_{\theta_1}, H^i(M)^{\theta_2}, H^i(M)_{\theta_2},$ $H^i(M'')^{\theta_3}, H^i(M'')_{\theta_3}$ are finitely generated $\mathbb{Z}_p$-modules. We have a long exact sequence
$$H^i (M') \xrightarrow{g_i} H^i(M) \xrightarrow{f_i} H^i (M'') \xrightarrow{b_i}H^{i+1}(M') \to \ldots $$ that commutes with the maps induced by the $\theta_i$s. It follows that (co)kernel of the maps $f_i, g_i, b_i$ are all equipped with an operator that we will simply denote as $\theta.$ Note that $\mathrm{Ker}(b_i)^\theta$ is finitely generated, as it embeds in $H^i(M'')^{\theta_3}$; further, $\mathrm{Ker}(b_i)_\theta$ is also finitely generated, since it receives a surjection from $H^{i}(M)_{\theta_2}.$ Similarly, one obtains an analogous claim for kernel of the maps $f_i, g_i$, as well as the cokernels. Note that we have a diagram
\begin{center}
\begin{tikzcd}
\mathrm{Ker}(b_i)^\theta \arrow[d] \arrow[r] & H^i(M'')^{\theta_3} \arrow[d] \\
\mathrm{Ker}(b_i)_\theta \arrow[r]           & H^i(M'')_{\theta_3}          
\end{tikzcd}
\end{center}where the top horizontal arrow is injective. By inverting $p$ and using our hypothesis, we obtain that $\mathrm{Ker}(b_i)^\theta [\frac 1 p] \to \mathrm{Ker}(b_i)_\theta [\frac 1 p]$ is injective. Similarly, we have a diagram
\begin{center}
\begin{tikzcd}
H^i(M)^{\theta_2} \arrow[d] \arrow[r] & \mathrm{Ker}(b_i)^\theta \arrow[d] \\
H^i(M)_{\theta_2} \arrow[r]           & \mathrm{Ker}(b_i)_\theta          
\end{tikzcd}
\end{center}which gives the surjectivity of $\mathrm{Ker}(b_i)^\theta [\frac 1 p] \to \mathrm{Ker}(b_i)_\theta [\frac 1 p]$. Thus it follows that the Bockstein characteristic of $(\mathrm{Ker}(b_i), \theta)$ exists. Similarly, one obtains an analogous claim for kernel of the maps $f_i, g_i$, as well as the cokernels. Let $\chi^u(b_i):= \chi^l(\mathrm{Bock}^\bullet(\mathrm{Ker}(b_i), \theta))$ and $\chi^v(b_i):= \chi^l(\mathrm{Bock}^\bullet(\mathrm{Coker}(b_i), \theta))$. Similarly, we define $\chi^u(f_i), \chi^u(g_i), \chi^v(f_i)$, and $ \chi^v(g_i)$. By \cref{lemm} and breaking the long exact sequence into short exact sequences it follows that 
$$\chi^l (\mathrm{Bock}^\bullet (H^i(M), \theta_1))= \chi^u(b_i) + \chi^v(b_{i-1}).$$ Similarly, we have $\chi^l (\mathrm{Bock}^\bullet (H^i(M''), \theta_3))= \chi^u(b_i) + \chi^v(f_{i}),$ and $\chi^l (\mathrm{Bock}^\bullet (H^i(M'), \theta_1))= \chi^v(b_{i-1}) + \chi^v(f_{i-1})$. Combining these, we obtain 
$$\chi^l (\mathrm{Bock}^\bullet (H^i(M''), \theta_3)) + \chi^l (\mathrm{Bock}^\bullet (H^i(M'), \theta_1)) = \chi^l (\mathrm{Bock}^\bullet (H^i(M), \theta_1)) + \chi^v(f_{i})+\chi^v(f_{i-1}).$$ Using \cref{alternating} and taking alternating sums, the result follows.
\end{proof}{}

We record a generalization of \cref{lemm}.

\begin{lemma}\label{uselater}
Let $M',M, M''$ be discrete $\mathbb{Z}_p$-modules. Suppose that we have a diagram
\begin{center}
   \begin{tikzcd}
M' \arrow[r] \arrow[d, "\theta_1"] & M \arrow[r] \arrow[d, "\theta_2"] & M'' \arrow[d, "\theta_3"] \\
M' \arrow[r]                       & M \arrow[r]                       & M''                      
\end{tikzcd} 
\end{center}{}    
such that $\mathrm{fib}(\theta_1) \to \mathrm{fib}(\theta_2) \to \mathrm{fib}(\theta_3)$ is a fiber sequence. Assume that $\mathrm{fib}(\theta_1)$ and $\mathrm{fib}(\theta_3)$ are perfect complexes, and the Bockstein characteristics of $(M', \theta_1), (M, \theta_2), (M'', \theta_3)$ exist. Then 
$$\chi^l (\mathrm{Bock}^\bullet(M, \theta_2))= \chi^l (\mathrm{Bock}^\bullet(M', \theta_1)) + \chi^l (\mathrm{Bock}^\bullet(M'', \theta_3)).$$ 

\end{lemma}{}
\begin{proof}
By our assumptions, $M'^{\theta_1}, M^{\theta_2}, M''^{\theta_3}, M'_{\theta_1}, M_{\theta_2}, M''_{\theta_3}$ are all finitely generated $\mathbb{Z}_p$-modules. The proof follows in a way exactly similar to the proof of \cref{lemm}. 
\end{proof}{}

\begin{corollary}\label{nir}
 Let $M$ be a discrete $\mathbb{Z}_p$-module. Suppose that $\theta: M \to M$ and $\theta': M \to M$ are two maps of $\mathbb{Z}_p$-modules. Suppose that $\mathrm{fib}(\theta), \mathrm{fib}(\theta')$ are perfect complexes and the Bockstein characteristics of $(M, \theta), (M, \theta')$, $(M, \theta \circ \theta')$ exist. Then 
$$\chi^l (\mathrm{Bock}^\bullet(M, \theta \circ \theta')) = \chi^l (\mathrm{Bock}^\bullet(M, \theta)) + \chi^l (\mathrm{Bock}^\bullet(M, \theta')). $$
 \end{corollary}{}

\begin{proof} Note that we have a diagram 
  \begin{center}
\begin{tikzcd}
M \arrow[d, "\theta'"] \arrow[r, "\mathrm{id}"] & M \arrow[d, "\theta \circ \theta'"] \arrow[r, "\theta'"] & M \arrow[d, "\theta"] \\
M \arrow[r, "\theta"]                           & M \arrow[r, "\mathrm{id}"]                               & M                    
\end{tikzcd}   
  \end{center} that induces a fiber sequence $\mathrm{fib}(\theta') \to \mathrm{fib}(\theta \circ \theta') \to \mathrm{fib}(\theta')$. Therefore the claim follows from \cref{uselater}.
  \end{proof}{}

\section{Stable Bockstein characteristic}\label{stablebock}
As the discussion after \cref{alternating} shows, existence of the Bockstein characteristic is a subtle matter, and not stable under extensions. In this section, we introduce one of the key new notions of this paper, which we call the stable Bockstein characteristic. Unlike the Bockstein characteristic, the stable Bockstein characteristic is defined for objects in a stable $\infty$-category satisfying very mild finiteness conditions, and agrees with the Bockstein characteristic when the latter exists.
The main result is the following.

\begin{proposition}\label{stabock}
 Let $M \in D^b (\mathbb{Z}_p)$ such that $M[\frac 1 p]$ is a perfect complex of $\mathbb{Q}_p$-vector spaces. Let $\theta: M \to M$ be a map in $D^b (\mathbb{Z}_p)$ such that $\mathrm{fib}(\theta)$ is a perfect complex over $\mathbb{Z}_p$. Then $\chi^l (\mathrm{Bock}^\bullet(M, \theta^r))$ exists for all $r \gg 0$. Moreover, $$\lim_{r \to \infty} \frac{\chi^l (\mathrm{Bock}^\bullet(M, \theta^r))}{r}$$ exists and is an integer.   
\end{proposition}{}

Before giving a proof of the proposition, let us introduce our desired definition.

\begin{definition}[Stable Bockstein characteristic]
 Let $M \in D^b (\mathbb{Z}_p)$ such that $M[\frac 1 p]$ is a perfect complex of $\mathbb{Q}_p$-vector spaces. Let $\theta: M \to M$ be a map in $D^b (\mathbb{Z}_p)$ such that $\mathrm{fib}(\theta)$ is a perfect complex over $\mathbb{Z}_p$. We define the stable Bockstein characteristic of $(M, \theta)$ to be $$\chi^{l}_{s}(\mathrm{Bock}^\bullet (M, \theta)):= \lim_{r \to \infty} \frac{\chi^l (\mathrm{Bock}^\bullet(M, \theta^r))}{r}.$$  
\end{definition}{}

\begin{proof}[Proof of \cref{stabock}]
By our hypothesis, it follows that $H^i(M)[\frac{1}{p}]$ is a finite dimensional $\mathbb{Q}_p$-vector space for every $i$. Further, using the short exact sequence from \cref{to}, it follows that $H^i(M)^\theta$ and $H^i(M)_\theta$ are finitely generated $\mathbb{Z}_p$-modules, implying that the two term complex $[H^i(M) \xrightarrow{\theta} H^i(M)]$ is perfect. Therefore, to prove the proposition, using \cref{alternating}, we can reduce to the case when $M$ is a discrete $\mathbb{Z}_p$-module. Let $V:= M[\frac{1}{p}]$ which is a finite dimensional $\mathbb{Q}_p$-vector space. Let $\theta':= \theta[\frac{1}{p}]$. We have a filtration $$\mathrm{Ker}(\theta') \subseteq \mathrm{Ker}(\theta'^2) \subseteq \ldots \subseteq \mathrm{Ker}({\theta'}^k)\subseteq \ldots \subseteq V;$$ we assume that it stabilizes at $\mathrm{Ker}({\theta'}^k)$.

By our choice of $k$, it follows that $0$ is a semisimple eignevalue for $\theta'^r$ for $r \ge k.$ By \cref{semisimple}, it follows that $\chi^l (\mathrm{Bock}^\bullet(M, \theta^r))$ exists for $r \ge k,$ proving the first part of our claim.

For the second part, we fix an $r \ge k.$ Using \cref{nir}, we can compute $\chi^l (\mathrm{Bock}^\bullet(M, \theta^{r(r+1)}))$ inductively in two different ways. Viewing $\theta^{r(r+1)}$ as $r$-fold composition of $\theta^{r+1}$, we see that $\chi^l (\mathrm{Bock}^\bullet(M, \theta^{r(r+1)}))=r \chi^l (\mathrm{Bock}^\bullet(M, \theta^{r+1}))$. Similarly, viewing $\theta^{r(r+1)}$ as $(r+1)$-fold composition of $\theta^r$, we get $\chi^l (\mathrm{Bock}^\bullet(M, \theta^{r(r+1)}))=(r+1) \chi^l (\mathrm{Bock}^\bullet(M, \theta^{r}))$. This implies 
\begin{equation}\label{d}
(r+1) \chi^l (\mathrm{Bock}^\bullet(M, \theta^{r}))= r \chi^l (\mathrm{Bock}^\bullet(M, \theta^{r+1})).    
\end{equation}{}
Since $\mathrm{gcd}(r, r+1)=1$, it follows that $r$ divides $ \chi^l (\mathrm{Bock}^\bullet(M, \theta^{r}))$. Therefore, using \cref{d}, we conclude that $\lim_{r \to \infty} \frac{\chi^l (\mathrm{Bock}^\bullet(M, \theta^r))}{r}$ exists and is an integer.
\end{proof}

The stable Bockstein characteristic has much better formal properties, as we note in the proposition below.

\begin{proposition}\label{be1}
Let $(M', \theta_1), (M, \theta_2)$ and $(M'', \theta_3)$ be such that we have a diagram
\begin{center}
\begin{tikzcd}
M' \arrow[r] \arrow[d, "\theta_1"'] & M \arrow[r] \arrow[d, "\theta_2"'] & M'' \arrow[d, "\theta_3"'] \\
M' \arrow[r]                       & M \arrow[r]                      & M''                       
\end{tikzcd}
\end{center}{}
 where $M'\to M \to M''$ is a fiber sequence in $D^b(\mathbb{Z}_p).$ Further, assume that $M'[\frac{1}{p}], M''[\frac{1}{p}]$ are perfect complexes over $\mathbb{Q}_p$ and $\mathrm{fib}(\theta_1), \mathrm{fib}(\theta_3)$ are perfect complexes over $\mathbb{Z}_p$. Then we have
$$\chi^l_s(\mathrm{Bock}^{\bullet}(M, \theta_2)) = \chi^l_s(\mathrm{Bock}^{\bullet}(M', \theta_1)) + \chi^l_s(\mathrm{Bock}^{\bullet}(M'', \theta_3)). $$
\end{proposition}

\begin{proof}
 Follows from \cref{stabock} by applying \cref{be} after replacing $\theta_i$ by $\theta_i^r$ for $r \gg 0.$
\end{proof}{}

\begin{example}
Let $p: \mathbb{Z}_p \to \mathbb{Z}_p$ be the multiplication by $p$ map. Then $\chi^l_s(\mathbb{Z}_p, p) = -1.$ The following result gives a broad generalization of this example.    
\end{example}{}

\begin{proposition}\label{uk}
 Let $M$ be a finitely generated $\mathbb{Z}_p$-module and $\theta: M \to M$ be an endomorphism. Let $\alpha_1, \ldots, \alpha_r \in \overline{\mathbb Q}_p$ be the nonzero eigenvalues of $\theta[\frac{1}{p}]$. Then $$\chi^l_s (\mathrm{Bock}^\bullet (M,\theta))= - v_p (\alpha_1\cdots  \alpha_r).$$ Note that if there are no nonzero eigenvalues, then the right hand side should be interpreted as $0$.   
\end{proposition}{}

\begin{proof}
 Since $M$ is finitely generated, the stable Bockstein characterstic exists. Note that $\prod_{i=1}^{r} \alpha_i \in \mathbb{Q}_p$, so the $p$-adic valuation is a well-defined integer. First we handle two cases.
\vspace{2mm}

\noindent
\textit{Case 1}. Suppose that $M$ is torsion. Since $M$ is finitely generated, it must be a finite length $\mathbb{Z}_p$-module and $\chi^l_s(M, \theta)$ is simply $\mathrm{length} (M)- \mathrm{length} (M) = 0$. In this case, there are no nonzero eignevalues and the claim in the proposition holds.
\vspace{2mm}

\noindent
\textit{Case 2}. Suppose that $M$ is a free $\mathbb{Z}_p$-module of finite rank and $\theta$ is injective. In this case, it follows that $\mathrm{Coker}(\theta)$ is finite length, and therefore, $\chi^l_s(M, \theta) = - \mathrm{length}(\mathrm{Coker}(\theta)).$ Let $A$ denote a matrix representing $\theta$. By the Smith normal form, we can write $A= SDR$, where $D$ is a diagonal matrix whose entries are given by $(\beta_1, \ldots. \beta_r)$, and $S,R$ are invertible matrices over $\mathbb{Z}_p$. It follows that $\mathrm{length}(\mathrm{Coker}(\theta)) = v_p(\beta_1 \cdots \beta_r).$ Note that have $A = (SR)R^{-1} D R.$ Therefore, $v_p(\det (A)) = v_p (\det(SR)) + v_p(\det (R^{-1}DR)) = v_p(\beta_1\cdots \beta_r).$ Combining the equalities gives the claim. 
\vspace{2mm}

Now we will handle the general case by induction on the rank of (the free part of) $M$. When rank is zero, the claim in the proposition holds by Case 1 above. Suppose that the claim is proven for any module whose rank is $\le n$. We will prove the claim when the rank of (the free part of) the given module $M$ is $n+1$.  We have an exact sequence $0\to T \to M \to F \to 0$, where $F$ is free and $T$ is torsion. Note that $\theta$ restricts to a map on $T$, that we denote by $\theta_T$, and there is an induced map $\theta': F \to F$ by passing to the quotient. By Case 1 above and \cref{be1}, we reduce to proving the claim for $(F, \theta')$. Now if $\mathrm{Ker}(\theta')$ is trivial, we are done by Case 2 above. Since $\mathrm{Ker}(\theta')$ is a submodule of $F$ it must be a free module; thus we may assume that $\mathrm{Ker}(\theta')$ has rank $\ge 1.$ Note that we have the following diagram
\begin{center}
    \begin{tikzcd}
0 \arrow[r] & \mathrm{Ker}(\theta') \arrow[d, "0"] \arrow[r] & F \arrow[r] \arrow[d, "\theta'"] & Q \arrow[d, "\theta''"] \arrow[r] & 0 \\
0 \arrow[r] & \mathrm{Ker}(\theta') \arrow[r]                & F \arrow[r]                      & Q \arrow[r]                       & 0
\end{tikzcd}
\end{center}
In the above, the quotient $F/\mathrm{Ker}(\theta')$ denoted by $Q$ is module whose rank (of the free part) is $\le n$. Since the claim in the proposition clearly holds for the zero map, we are done by \cref{be1} and induction.
\end{proof}{}

\section{Special values of zeta functions}
In this section, we prove the main theorem of our paper regarding special values of zeta functions (\cref{special}). In order to state our main result, we will need a few definitions. To this end, we fix some notations for this section. Let $k$ be the finite field $\mathbb{F}_q$, where $q= p^n$. Let $M$ be a dualizable object in $\FGauge(k).$
\begin{construction}[Hodge--realization]\label{hodgerea}
 In the above scenario, we can define the Hodge realization of $M$, which gives a graded object $M^{\mathrm{Hodge}} = \bigoplus_{i} M^{\mathrm{Hodge}, i}$ in the category of perfect complexes over $k$. In the language of the stack $(\mathrm{Spec}\,k)^\mathrm{syn}$, this can be defined by taking pullback of $M$ along the closed immersion $B\mathbb{G}_m \to (\mathrm{Spec}\,k)^\mathrm{syn}$.

 We define the $j$-th Hodge cohomology group of $M$ to be $H^{i+j} (M^{\mathrm{Hodge},i}),$ which is a finite dimensional $k$-vector space, and will be denoted by $H^{i,j}_{\mathrm{Hodge}}(M).$
\end{construction}

\begin{definition}
The Hodge number of $M$, denoted as $h^{i,j}(M)$, is defined to be $\dim_k H^{i,j}_{\mathrm{Hodge}}(M).$ Note that these numbers are zero for all but finitely many $i,j \in \mathbb{Z}$.    
\end{definition}

\begin{definition}[Weighted Hodge--Euler characteristic]\label{whe}
For a dualizable $F$-gauge $M$, the $r$-th weighted Hodge Euler characteristic denoted by $\chi(M,r)$ is defined to be $$\chi(M,r):= \sum_{\substack{i,j \in \mathbb{Z}, \\ i \le r}} (-1)^{i+j} (r-i)h^{i,j}(M).$$  
\end{definition}{}

The following is the key new definition, that uses the stable Bockstein characteristic introduced in this paper.
\begin{definition} Let $\gamma$ denote the operator coming from Galois action on $R\Gamma_{\mathrm{syn}}(\overline{M}, \mathbb{Z}_p(r))$. Then the stable Bockstein characteristic ${\chi^l_s(\mathrm{Bock}^\bullet (R\Gamma_{\mathrm{syn}}(\overline{M}, \mathbb{Z}_p(r)), \gamma-1))}$ is well-defined, since $R\Gamma_{\mathrm{syn}}(\overline{M}, \mathbb{Z}_p(r))[\frac{1}{p}]$ is a perfect complex over $\mathbb{Q}_p$ (see \cref{corol}) and $\mathrm{fib}(\gamma - 1) \simeq R\Gamma_{\mathrm{syn}}({M}, \mathbb{Z}_p(r))$ is a perfect complex over $\mathbb{Z}_p$. We define $$\mu_{\mathrm{syn}}(M, r):= p^{{\chi^l_s(\mathrm{Bock}^\bullet (R\Gamma_{\mathrm{syn}}(\overline{M}, \mathbb{Z}_p(r)), \gamma-1))}}.$$ Roughly speaking, one may think of $\mu_{\mathrm{syn}}(M,r)$ as the ``size" of $R\Gamma_{\mathrm{syn}}(M, \mathbb{Z}_p(r))$ in a special sense.  
\end{definition}
Now we can state our main result regarding special values of the Zeta function of $M$. Below, for an element $t \in K$, we set $|t|_p := (\frac{1}{p})^{v_p(t)},$ which is called the (normalized) $p$-adic norm.

\begin{theorem}\label{special}
Let $r \in \mathbb{Z}$. Suppose that $\rho$ is the order of the zero of $\zeta (M, s)$ at $s=r$. Then 
$$\left| \lim_{s \to r} \frac{\zeta (M,s)}{(1-q^{r-s})^\rho}\right| _p = \frac{1}{\mu_{\mathrm{syn}}(M,r)q^{\chi(M,r)}}.$$ 
\end{theorem}

An important preliminary observation in our proof is that the weighted Hodge Euler characteristic can be understood in a different way from the $F$-gauge $M$. To this end, let us review certain piece of structures that are present in an $F$-gauge. By pulling back $M$ along the map $(\mathrm{Spec}\, k)^{\mathcal{N}} \to (\mathrm{Spec}\, k)^{\mathrm{syn}}$, we can obtain a filtered module (over the filtered ring $W(k)$, equipped with the $p$-adic filtration) which we denote as $\mathrm{Fil}^\bullet M.$ The graded pieces of this filtration will be denoted as $\mathrm{gr}^\bullet M$. Let $M^u$ denote the $W(k)$-module underlying the $F$-gauge. There is a natural map $\mathrm{Fil}^r M \to M^u$ whose cofiber will be denoted by $M^u/ \mathrm{Fil}^r M.$ Since $M$ is a dualizable $F$-gauge, it follows that $M^u/ \mathrm{Fil}^r M$ is a perfect complex of $\mathbb{Z}_p$-modules. Further, we have that $(M^u/ \mathrm{Fil}^r M)[\frac{1}{p}] \simeq 0$. In other words, each cohomology groups of $M^u/ \mathrm{Fil}^r M$ must be a finite length $\mathbb{Z}_p$-module. In particular, $\chi^l(M^u/ \mathrm{Fil}^r M)$ is well-defined.
\begin{definition}
 In the above scenario, the integer $\chi^l(M^u/ \mathrm{Fil}^r M)$ will be called the $r$-th Nygaard characteristic of the dualizable $F$-gauge $M$.   
\end{definition}

Now we are ready to state and prove the following proposition, which gives an interpretation of the classical Hodge Euler characteristic in terms of the Nygaard characteristic that we just defined.

\begin{proposition}\label{niceob}
Let $M$ be a dualizable $F$-gauge over $k.$ Then
$$p^{\chi^l (M^u/\mathrm{Fil}^r M)}=q^{\chi(M,r)}.$$   
\end{proposition}{}
\begin{proof}
This amounts to the assertion that $\chi^l (M^u/\mathrm{Fil}^r M)=n \cdot \chi(M,r)$. Since $M$ is a dualizable $F$-gauge, $M^u/\mathrm{Fil}^r M$ admits a finite filtration whose graded pieces are given by $\mathrm{gr}^i (M)$ for $i \le r-1.$ Note that $\mathrm{gr}^i(M)$ is also a perfect complex of $\mathbb{Z}_p$-modules whose cohomology groups are finite length $\mathbb{Z}_p$-modules. Therefore,
\begin{equation}\label{eq1}
 \chi^l(M^u/\mathrm{Fil}^rM) = \sum_{i \le r-1} \chi^l (\mathrm{gr}^i(M)).   
\end{equation}

Note that there is a closed immersion $$h_{HT,+}: \mathbb{A}^1_k/ \mathbb{G}_m \to (\mathrm{Spec}\, k)^\mathrm{syn},$$ and pullback of $M$ along this map can be thought of as a filtered $k$-modules which we may write as $\mathrm{Fil}^\bullet M^{HT}.$ The associated graded of this filtered object will be denoted as $\mathrm{gr}^\bullet M^{HT}$. Forming the associated graded in this case amounts to pulling back along the map $B\mathbb{G}_m \to (\mathrm{Spec}\,k)^{\mathrm{syn}}$. This implies that we must have
$$\oplus_i \mathrm{gr}^i M^{HT} \simeq \oplus_i M^{\mathrm{Hodge}, i}.$$
 By construction, it also follows that $\mathrm{gr}^i(M) \simeq \mathrm{Fil}^i M^{HT}.$ One can again put a finite filtration on $\mathrm{Fil}^i M^{HT}$ such that the graded pieces are given by $\mathrm{gr}^j M^{HT}$ for $j \le i$. This implies that 
\begin{equation}\label{eq2}
\chi^l( \mathrm{Fil}^i M^{HT}) = \sum_{j \le i} \chi^l (M^{\mathrm{Hodge},j}).   
\end{equation}
Further, one has 
\begin{equation}\label{eq3}
 \chi^l (M^{\mathrm{Hodge},j}) = \sum_{k} (-1)^{j+k} n\cdot h^{j,k}(M). 
\end{equation}
Combining these equalities, we obtain that
$$\frac{\chi^l (M^u/\mathrm{Fil}^r M)}{n} = \sum_{\substack{i \le r-1\\ j \le i \\ k \in \mathbb{Z}}} (-1)^{j+k} h^{j,k}(M).$$ However, each summand appears exactly once for every value of $i \in \left \{j, j+1, \ldots, r-1 \right \}$. The latter set has size $r-j$. Therefore, the above sum is equal to 
$$\sum_{\substack{j,k \in \mathbb{Z}\\j \le r}} (-1)^{j+k} (r-j) h^{j,k}(M).$$This finishes the proof.
\end{proof}

\begin{remark}[Twists and zeta values]\label{twistsz}
 Suppose that $V$ is an isocrystal over $\mathbb{F}_{q}$, where $q = p^n$. Let $V \left \{i \right \}$ denote the $i$-th Breuil--Kisin twist of $V$. Let $P_V(t) = \det (1- tF^n|V)$ and $P_{V \left \{i \right \}}(t) = \det (1- tF^n|V \left \{i \right \}).$ Then it follows that $P_{V \left \{i \right \}}(t) = P_V(q^{-i}t)$. This implies that if $M$ is a dualizable $F$-gauge over $k$, and $M \left \{i\right \}$ is the $i$-th Breuil--Kisin twist of $M$ as an $F$-gauge, then we have
$$\zeta(M\left \{i \right \}, s)=\zeta(M, s+i). $$ Moreover, by construction, it follows that we have $${\chi^l_s(\mathrm{Bock}^\bullet (R\Gamma_{\mathrm{syn}}(\overline{M \left \{i \right \}}, \mathbb{Z}_p(r)), \gamma-1))}= {\chi^l_s(\mathrm{Bock}^\bullet (R\Gamma_{\mathrm{syn}}(\overline{M}, \mathbb{Z}_p(r+i)), \gamma-1))}.$$ Thus we have
$$\mu_{\mathrm{syn}} (M \left\{i \right \}, r) = \mu_{\mathrm{syn}} (M , r+i).$$ By \cref{niceob}, it also follows that $$\chi(M\left \{i \right\},r)= \chi(M,r+i). $$ Consequently, proving \cref{special} for a fixed $M$ and a fixed $r \in \mathbb{Z}$ is equivalent to proving \cref{special} for $M \left\{i \right \}$ and $r-i$. 
\end{remark}{}

Note that viewing an $F$-gauge $M$ as an object of the derived category of the stack $(\mathrm{Spec}\,k)^{\mathrm{syn}}$, one has 
$R\Gamma_{\mathrm{syn}}(M, \mathbb{Z}_p(r)) \simeq R\Gamma ((\mathrm{Spec}\,k)^{\mathrm{syn}}, M \left \{r\right \} ).$ When $M$ is the $F$-gauge associated to a smooth proper scheme $X$ over $k$, this gives a new way (involving the Nygaard filtration) of understanding weight $n$ syntomic cohomology of $X$ which is quite different from the classical one involving de Rham--Witt complexes. This idea, along with \cref{niceob}, plays a key role in the proof of the proposition below.
\begin{proposition}\label{l0}
Let $M$ be a dualizable $F$-gauge such that for some $m \in \mathbb{N}$, multiplication by $p^m$ on $M$ is homotopic to zero. Then we have
$$\mu_{\mathrm{syn}} (M,r)= \frac{1}{q^{\chi(M,r)}}.$$    
\end{proposition}{}

\begin{proof}
Note that under our assumptions, $R\Gamma_{\mathrm{syn}}(M,\mathbb{Z}_p(r))$ is a perfect complex of $\mathbb{Z}_p$-modules, whose cohomology groups are all killed by $p^m$. Therefore, the cohomology groups are all finite length $\mathbb{Z}_p$-modules. In this case, the stable Bockstein characteristic ${\chi^l_s(\mathrm{Bock}^\bullet (R\Gamma_{\mathrm{syn}}(\overline{M}, \mathbb{Z}_p(r)), \gamma-1))}$ is simply equal to $\chi^l (R\Gamma_{\mathrm{syn}}(M,\mathbb{Z}_p(r))).$ Using the fact that $R\Gamma_{\mathrm{syn}}(M, \mathbb{Z}_p(r)) \simeq R\Gamma ((\mathrm{Spec}\,k)^{\mathrm{syn}}, M \left \{r\right \} )$, we obtain a fiber sequence
$$R\Gamma_{\mathrm{syn}}(M, \mathbb{Z}_p(r) ) \to \mathrm{Fil}^r M \xrightarrow{\varphi - \mathrm{can}} M^u,$$ where $\varphi$ is a certain Frobenius linear map and $\mathrm{can}$ is the canonical map arising from the filtered object $\mathrm{Fil}^\bullet M.$ It follows from our assumptions that $\chi^l (\mathrm{Fil}^r M)$ and $\chi^l ( M^u)$ both exist and 
\begin{equation}\label{lw1}
 \chi^l (\mathrm{Fil}^r M) - \chi^l ( M^u) = \chi^l (R\Gamma_{\mathrm{syn}}(M,\mathbb{Z}_p(r))).  
\end{equation}
Also, we have a fiber sequence 
$$\mathrm{Fil}^r M \xrightarrow{\mathrm{can}} M^{u} \to M^u/\mathrm{Fil}^r M. $$ It follows that one has
\begin{equation}\label{lw2}
\chi^l (\mathrm{Fil}^r M) - \chi^l ( M^u) =- \chi^l(M^u/ \mathrm{Fil}^r M).  
\end{equation}
Combining the two equalities above, we obtain $$\chi^l (R\Gamma_{\mathrm{syn}}(M,\mathbb{Z}_p(r)))= - \chi^l(M^u/ \mathrm{Fil}^r M).$$ By \cref{niceob}, we obtain $$\mu_{\mathrm{syn}}(M,r)= \frac{1}{q^{\chi(M,r)}} ,$$ as desired.
\end{proof}

Now we will establish certain structural results regarding dualizable $F$-gauges over $k$.

\begin{lemma}\label{l1}
Let $M$ be a dualizable $F$-gauge over $k.$ Then there exists a fiber sequence
$$  T \to M \to V $$such that $V \simeq \oplus_{i=1}^{m} V_i[r_i]$, where $V_i$ is a vector bundle on $(\mathrm{Spec}\, k)^{\mathrm{syn}}$, $r_i \in \mathbb{Z},$ and $T$ is an $F$-gauge such that multiplication by $p^k$ is null-homotopic for some $k \in \mathbb{N}.$ 
\end{lemma}{}

\begin{proof}
 We will use the fact that $\mathrm{Perf}((\mathrm{Spec}\, k)^\mathrm{syn}) \otimes{\mathbb{Q}_p} \simeq \mathrm{Perf}(\mathrm{Spec}\,K)^\mathrm{syn}$ as defined in \cref{co1}. The image of $M$ in $\mathrm{Perf}((\mathrm{Spec}\, k)^\mathrm{syn}) \otimes{\mathbb{Q}_p}$ will be denoted by $M[1/p].$ Under the above isomorphism, $M[1/p]$ can be regarded as a perfect complex of $k$-vector spaces equipped with a Frobenius semi-linear isomorphism. Since the ring Frobenius twisted polynomial ring $k_{\sigma}[F]$ is left-hereditary, one has an isomorphism $$f': M[1/p] \simeq \bigoplus_{i=1}^{m} V'_i[r_i]$$ in $\mathrm{Perf}(\mathrm{Spec}\,K)^\mathrm{syn}$, where for each $i$, $V'_i$ is an isocrystal over $k.$ By choosing a lattice for $V'_i$, one can lift $V'_i$ to a vector bundle on $(\mathrm{Spec}\,k)^{\mathrm{syn}}$, which we will denote by $V_i$. Let us set $V := \oplus_{i=1}^{m} V_i[r_i]$. By replacing $f'$ by $p^kf'$ for $k \gg 0$, we can without loss of generality assume that there exists a map 
$f: M \to V $ in $\mathrm{Perf}((\mathrm{Spec}\, k)^{\mathrm{syn}})$ that induces $f'$. Let $T:= \mathrm{fib}(M \xrightarrow{f} N).$ By construction, it follows that $T[1/p] \simeq 0$ in $\mathrm{Perf}(\mathrm{Spec}\,K)^\mathrm{syn}.$ The latter implies that the zero map on $T$ is homotopic to multiplication by a power of $p.$
\end{proof}

\begin{remark}\label{identify}
Suppose that $k$ is a perfect field. Then a vector bundle on $(\mathrm{Spec}\,k)^\mathrm{syn}$ with Hodge--Tate weights in the integers $\left \{0, 1\right\}$ can be concretely identified with finite free $W(k)$-modules $M$ with a Frobenius semilinear operator $F: M \to M$ such that $p M \subseteq F(M)$ (see \cite[Prop.~3.45]{Mon}). We will call the category of such objects Dieudonn\'e modules. By Dieudonn\'e theory, the category of such vector bundles is equivalent to the category of $p$-divisible groups over $k.$
\end{remark}{}

\begin{definition}
 Let $k$ be a perfect field (not assumed to be algebraically closed) and $i \in \mathbb{Q}$. An isocrystal $V$ over $k$ will be called pure of slope $i$ if the base change of $V$ viewed as an isocrystal over $\overline{k}$ is a direct sum of finitely many copies of $E_i$.  
\end{definition}

\begin{lemma}\label{se}
Let $i \in \mathbb{Q} \cap [0,1].$ Let $V$ be an isocrystal over a perfect field $k$ that is pure of slope $i$. Then there exists a Dieudonn\'e module $(M,F)$ such that $V$ is isomorphic to the isocrystal corresponding to $(M[\frac{1}{p}],F[\frac{1}{p}]).$    
\end{lemma}

\begin{proof}
Let us first suppose that $k$ is algebraically closed. Let $i = s/r$ where $\gcd(s,r)=1$ and $r>0$. By the Dieudonn\'e--Manin classification, it follows that $V$ is a direct sum of finitely many copies of  $ E_{s/r}.$ Therefore, in this case, we may assume without loss of generality that $V= E_{s/r}$. Let $K = W(k) [\frac{1}{p}].$ By \cref{ex}, one can identify the isocrystal $V$ as a semilinear map $F: K^r \to K^r$ given by 
$$F(x_1, \ldots, x_r) := (\varphi_K(x_2), \ldots, \varphi_K (x_{r}), p^s\varphi_K(x_1)).$$  Note that $s \le r.$ If $0 \le s \le 1$, one can simply take $M$ to be the $F$-stable lattice $W(k)^{\oplus r}$; in this case the condition $p M \subseteq F(M)$ is clearly satisfied. Now we may assume that $s \ge 2$. Let $e_1, \ldots, e_r$ denote the standard $K$-basis vectors of $K^r.$ We define $M$ to be the $W(k)$-span of the elements in the set
$$\left \{\frac{e_1}{p^{s-1}}, \frac{e_2}{p^{s-2}}, \ldots, e_s, \ldots, e_r \right \} .$$ It follows by construction that $M$ is an $F$-stable lattice and $p M \subseteq F(M).$ 

Now we return to the case when $k$ is only assumed to be perfect and let $\overline{k}$ denotes an algebraic closure. We will argue by Galois descent. Let $\overline{V}$ denote the base change of $V$, viewed as an isocrystal over $\overline{k}.$ Let $\gamma: \overline{V} \to \overline{V}$ denote the Galois action. By the above paragraph, one obtains a Dieudonn\'e module $(M,F)$ corresponding to the isocrystal $\overline{V}.$ Since $M$ is a lattice in $\overline{V}$, the Galois action induces a map 
$\gamma: M \to \frac{M}{p^k}$ for some $k \gg 0.$ Let $\iota: M \to \frac{M}{p^k}$ denote the natural inclusion. We set $N:= \mathrm{Ker}(\gamma - \iota).$ It follows that $N [\frac{1}{p}] \simeq \overline{V}^{\gamma} \simeq V$. Thus $N$ is a lattice in $V$. Let $F$ and $\overline{F}$ denote the Frobenius semilinear endomorphisms of $V$ and $\overline{V}$ respectively. Since $\overline{F}$ and $\gamma$ commutes, and $M$ is $\overline{F}$-stable, it follows that $N$ is $F$-stable. It would be enough to show that $p N \subseteq F(N)$. To this end, let $x \in N.$ Then $p x = \overline{F} (y)$ for some $y \in \overline{M}.$ Therefore, $$\overline{F}(y)= px = p \gamma (x) = \gamma \overline{F}(y) = \overline{F}(\gamma (y)).$$ Since $\overline{F}$ acts bijectively on $\overline{V}$, the above equation implies that $\gamma(y) = y$; i.e., $px = F(y)$ for $y \in N$. This finishes the proof.
\end{proof}{}

\begin{lemma}\label{l2}
Let $V$ be a vector bundle on $(\mathrm{Spec}\,k)^\mathrm{syn}.$ Then there exists a fiber sequence 
$$T \to V \to \bigoplus_{i=-m_1}^{m_2} U_i$$ such that $U_i \left \{i \right\}$ is a vector bundle with Hodge--Tate weights in $\left \{0,1\right \},$ and $T$ is an $F$-gauge such that multiplication by $p^k$ is null-homotopic for some $k \in \mathbb{N}.$ 
\end{lemma}

\begin{proof}
 Our proof will use the notion of slopes. Let $V'$ denote the isocrystal over $k$ given by $V[1/p].$ Let $\overline{V'}$ the isocrystal over $\overline{k}$ obtained by base change. By the Dieudonn\'e--Manin classification one can write $\overline{V'} = \bigoplus W_i$ where $W_i$ is pure of slope $i$ and all but finitely many of the $W_i$-s are zero. Note that $\overline{V'}$ has a natural Galois action that preserves $W_i$ for all $i$. Therefore, by descent, we obtain a direct sum decomposition $V' = \bigoplus V_i$ such that each $V_i$ is pure of slope $i$ and all but finitely many of them are zero. Let $a_i$ be the unique integer such that $i \in [a_i, a_i+1)$. Then $V_i \left \{a_i \right \}$ is pure of slope $i - a_i$ which is a rational number in the interval $[0,1].$ Let $M_i$ denote the Dieudonn\'e module such that $M_i [\frac{1}{p}] \simeq V_i \left \{a_i \right \}$, as can be chosen by \cref{se}. Therefore, $M_i \left \{-a_i \right \}$ is a vector bundle on $(\mathrm{Spec}\, k)^\mathrm{syn}$ whose isogeny class is isomorphic to $V_i$. In other words, there is an isomorphism 
$$f':V[1/p] \simeq \oplus M_i \left \{-a_i \right \} [1/p].$$ By replacing $f'$ by $p^mf$ for $m \gg 0$ if necessary, we can assume without loss of generality that $f'$ is induced by a map $f:  V \to \oplus M_i \left \{-a_i \right \} $ in $\mathrm{Perf}((\mathrm{Spec}\, k)^\mathrm{syn}).$ Let $T= \mathrm{fib}(f).$ Since $T[1/p] \simeq 0$ in $\mathrm{Perf}((\mathrm{Spec}\, k)^\mathrm{syn}) \otimes \mathbb{Q}_p$ it follows that the zero map on $T$ is homotopic to multiplication by a power of $p$. By changing the indexing, we obtain the claim.
\end{proof}

As we have remarked earlier, in general for an $F$-gauge $M$ over an algebraically closed field $k$, the groups $H^i(M, \mathbb{Z}_p(r))$ need not be finitely generated $\mathbb{Z}_p$-modules. However, in the proposition below, we prove that this issue does not occur when $M$ is a vector bundle on $(\mathrm{Spec}\, k)^{\mathrm{syn}}$ with Hodge--Tate weights in $\left \{0,1 \right \}.$ As we will see, this will require an application of Dieudonn\'e theory in the language of $F$-gauges as developed and studied in \cite{Mon}.

\begin{proposition}\label{ts}Let $k$ be an algebraically closed field. Let $M$ be a vector bundle on $(\mathrm{Spec}\,k)^\mathrm{syn}$ of Hodge--Tate weights in $\left \{0,1\right \}$. Then $H^i (M, \mathbb{Z}_p(r))=0$ whenever $i \ne 0$ or $r \ne \left\{0,1\right\}.$ Moreover, $H^0 (M, \mathbb{Z}_p(r))$ is a finitely generated $\mathbb{Z}_p$-module for $0 \le r \le 1.$   
\end{proposition}{}

\begin{proof}
It follows from the definitions that $H^{<0} (M, \mathbb{Z}_p(r))=0$.    

Note that we may also regard $M$ as a Dieudonn\'e module (\cref{identify}). Let $M^u$ denote the free $W(k)$-module of finite rank, equipped with a Frobenius map $F: M^u \to M^u.$ For $r <0 $, $H^* (M, \mathbb{Z}_p(r))$ is computed as cohomology of the two term complex $M^u \xrightarrow{p^{-r}F - \mathrm{id}} M^u$, where the source is in degree $0.$ However, viewed as an object of $D(\mathbb{Z}_p)$, this complex is derived $p$-complete, and is isomorphic to $0$ modulo $p.$ This shows that the complex $M^u \xrightarrow{p^{-r}F - \mathrm{id}} M^u$ is acyclic. Therefore, $H^i (M, \mathbb{Z}_p(r))=0$ for all $i$ and $r <0$.

For $r \ge 0$, $H^* (M, \mathbb{Z}_p(r))$ is computed as cohomology of the two term complex 
\begin{equation}\label{co}
   \mathrm{Fil}^r M^u \xrightarrow{p^{-r}F - \iota} M^u 
\end{equation}where $\mathrm{Fil}^r M^u$ denotes $r$-th step of the Nygaard filtration. Therefore, it follows that $H^{>1} (M, \mathbb{Z}_p(r) )=0$.

Now suppose that $r>1$. We will first show that $H^0 (M, \mathbb{Z}_p(r))=0.$ Since $M^u$ is a free $W(k)$-module, it follows that $H^0 (M, \mathbb{Z}_p(r))$ is $p$-torsion free. Therefore, it is enough to show that $H^0 (M, \mathbb{Z}_p(r))[1/p]=0.$ However, note that the slopes of the isocrystal $M^u[1/p]$ are rational numbers in the interval $[0,1].$ By the Dieudonn\'e--Manin classification, there is no nonzero map $\mathcal{O}\left \{-r\right \} \to M^u[1/p]$ for $r>1.$ Therefore, by \cref{isom}, we obtain $H^0 (M, \mathbb{Z}_p(r))[1/p]=0$, as desired. In order to show that $H^1 (M, \mathbb{Z}_p(r))=0$, we will use \cite{Mon}. By \cite[Thm.~1.11]{Mon}, $M$ corresponds to $F$-gauge associated to a $p$-divisible group $G$, and $\mathrm{Fil}^* M^u \simeq \mathrm{Fil}^*_{\mathrm{Nyg}} H^2_{\Prism }(BG)$. Therefore, we have a decomposition $M^u\simeq T \oplus W$ of $W(k)$-modules such that for $i \ge 1$,
$$\mathrm{Fil}^i M^u \simeq p^i T \oplus p^{i-1} W.$$ Note that there is an isomorphism $T \oplus W \to p^i T \oplus p^{i-1}W$ that sends $(x,y) \mapsto (p^ix, p^{i-1}y). $ Let us denote the map $T \oplus W \to T \oplus W $ that sends $(x,y) \mapsto (p^ix, p^{i-1}y)$ by $s_r$. We also have a map $T \oplus W \to T \oplus W$ that is determined by $(x,y) \mapsto F(x) + \frac{F(y)}{p} \in M^u;$ we denote this map by $\Phi.$ In order prove that $H^i (M, \mathbb{Z}_p(r)) =0$ for $r>1$, we need to prove that the map $  \mathrm{Fil}^r M^u \xrightarrow{p^{-r}F - \iota} M^u $ is surjective. Under the above identifications, it suffices to prove that the map 
$$T \oplus W \xrightarrow{\Phi - s_r} T \oplus W$$is surjective. By $p$-completeness, surjectivity can be checked modulo $p.$ For $r>1$, modulo $p$, the map $s_r =0.$ Therefore, for $r>1$, to show surjectivity of $(\Phi - s_r)$, it is enough to show surjectivity of $\Phi.$ Note that $\Phi$ is the composition $$T \oplus W \xrightarrow{(x,y) \mapsto (px, y)} pT \oplus W \simeq \mathrm{Fil}^1 M^u \xrightarrow{p^{-1}F} M^u.$$
Now the desired surjectivity follows from the surjectivity of $p^{-1}F$. The latter is true because $M^u$, naturally equipped with the structure of a Dieudonn\'e module, satisfies $pM^u \subseteq F(M^u)$.

Now suppose that $r=0$. Then $H^* (M, \mathbb{Z}_p)$ is computed as the cohomology of the complex $M \xrightarrow{F - \mathrm{id}} M$, where the source of the map is in degree $0.$ Note that $F- \mathrm{id}$ is surjective, since it is surjective modulo $p$ and $M$ is $p$-complete. This implies that $H^1 (M, \mathbb{Z}_p)=0.$ Since $M$ is $p$-torsion free, it follows that $H^0 (M, \mathbb{Z}_p)/p$ is kernel of the map $M/p 
\to M/p$ induced by $F- \mathrm{id}$. Since $M/p$ is a finite dimensional $k$-vector space, it follows from \cite[Lem.~4.1.1]{MR23} that $H^0 (M, \mathbb{Z}_p)/p$ is a finite dimensional $\mathbb{F}_p$-vector space. Since $H^0 (M, \mathbb{Z}_p)$ is $p$-complete, it follows that $H^0 (M, \mathbb{Z}_p)$ is a finitely generated $\mathbb{Z}_p$-module.

Finally, we handle the case when $r=1$. Let $G$ be the $p$-divisible group associated to the $F$-gauge $M$. By \cite[Prop.~3.15]{Mon}, we have 
$$H^*_{(k)_\mathrm{qsyn}} (k, T_p(G^\vee)) \simeq H^* (M, \mathbb{Z}_p(1)),$$ where $G^\vee$ is the dual of $G$ and $T_p(G^\vee)$ is the Tate module. In particular, we have $H^0 (M, \mathbb{Z}_p(1)) \simeq T_p(G^\vee)(k).$ Regarding $G^\vee$ as a sheaf on the quasisyntomic site of $k$, we have an isomorphism $G^\vee[p] \simeq T_p(G^\vee)/p$ of (pre)sheaves. This implies that $(T_p(G^\vee)(k))/p \simeq G^\vee [p] (k).$ Since $G^\vee[p]$ is a finite commutative group scheme of rank $p$ over $k$,we see that $G^\vee[p](k)$ is a finite dimensional $\mathbb{F}_p$-vector space. Since $T_p(G^\vee)(k)$ is $p$-complete, we conclude that it must be a finitely generated $\mathbb{Z}_p$-module. This proves that $H^0 (M, \mathbb{Z}_p(1))$ is a finitely generated $\mathbb{Z}_p$-module. To prove $H^1(M, \mathbb{Z}_p(1))=0$ it suffices to prove that $H^1_{(k)_\mathrm{qsyn}}(k, T_p(G^\vee))=0.$ However, this follows from \cref{lemmav}
\end{proof}

\begin{lemma}\label{lemmav}
 Let $T$ be a finite, commutative group scheme of $p$-power rank over an algebraically closed field $k$. Then $H^{>0} _{(k)_\mathrm{qsyn}}(k, T)=0$.   
\end{lemma}
\begin{proof}
Since $k$ is algebraically closed, by devissage we may reduce to the case when $H$ is either $\alpha_p$, $\mathbb{Z}/p$ or $\mu_p$. The case of $\alpha_p$ and $\mathbb{Z}/p$ follow from the exact sequences (of quasisyntomic sheaves)
$$0 \to \alpha_p \to \mathbb{G}_a \xrightarrow{\mathrm{Frob}} \mathbb{G}_a \to 0,$$ and $$0 \to \mathbb{Z}/p \to \mathbb{G}_a \xrightarrow{\mathrm{Frob} - \mathrm{id}} \mathbb{G}_a \to 0.$$ For $\mu_p$, we argue as follows. Note that there is a fiber sequence
$$R\Gamma _{(k)_\mathrm{qsyn}}(k, \mu_p) \to R\Gamma (k, \mathbb{Z}_p(1)) \xrightarrow{p} R\Gamma (k, \mathbb{Z}_p(1)).$$ However, concretely, $H^*(k, \mathbb{Z}_p(1))$ can be computed as cohomology of the complex $$pW(k) \xrightarrow{p^{-1} F - \iota} W(k),$$ or equivalently, as cohomology of $$W(k) \xrightarrow{F - p} W(k).$$ However, this complex is acyclic, as can be checked by reducing modulo $p$. Therefore, $R\Gamma_{(k)_\mathrm{qsyn}}(k, \mu_p) \simeq 0.$ This finishes the proof.
\end{proof}{}

\begin{proposition}\label{coconut}
Let $G$ be a $p$-divisible group over a perfect field $k$. Let $M(G)$ denote the Dieudonn\'e module of $G$. Then $$\dim G= \sum \lambda_i,$$ where $\lambda_i$ is the set of slopes of the isocrystal $M(G)[1/p]$, counted with multiplicities. 
\end{proposition}{}

\begin{proof}
 It suffices to prove this after base change. Therefore, we may assume that $k$ is algebraically closed. Note that the dimension of a $p$-divisible group is invariant under isogeny. Therefore, by the Dieudonn\'e--Manin classification, we can reduce to checking the proposition for a $p$-divisible group $G$ such that $M(G)[1/p]$ is isomorphic to $E_{s/r}$ as an isocrystal over $k$ for coprime integers $r,s$ such that $ 0 \le s \le r$ and $r>0$. For example, one may take a $G$ for which $M(G) \simeq \mathscr{D}_k/ \mathscr{D}_k (F^{r-s} - V^{s})$, where $\mathscr{D}_k$ denotes the Dieudonn\'e ring. In this case, it follows that $\dim G = \dim_k M(G)/FM(G) = s.$ However, $E_{s/r}$ only has $s/r$ as a slope, with multiplicity $r$. Thus sum of all slopes with multiplicity equals $s$ as well. This finishes the proof.  
\end{proof}

\begin{remark}\label{hn}
For every $p$-divisible group $G$ over a perfect field $k$, we have an exact sequence
$$0 \to t_{G^\vee} \to M(G)/p \to w_G \to 0,$$ which is called the Hodge--Tate sequence (see \cite[Prop.~3.51]{Mon}). In particular, by viewing $M(G)$ as a vector bundle on $(\mathrm{Spec}\, k)^{\mathrm{syn}}$, the pullback along $B\mathbb{G}_m \to (\mathrm{Spec}\, k)^{\mathrm{syn}}$ identifies with $t_{G^\vee} \oplus \omega_G$ as a graded vector space, where $t_{G^\vee}$ has weight $0$ and $\omega_G$ has weight $1$. The only nonzero Hodge--numbers can be described as follows. We have $h^{0,0}= \dim_k t_{G^\vee} = \dim G^\vee$ and
$h^{1,-1} = \dim \omega_G = \dim G.$    
\end{remark}

\begin{lemma}\label{usefulll}
 In the set up of \cref{hn}, the weighted Hodge--Euler characteristic $\chi(M(G),r)=0$ for $r \le 0$. For $r \ge 1$, we have  $$\chi(M(G), r)= r \dim G^\vee + (r-1) \dim G. $$  
\end{lemma}{}
 
\begin{proof}
  Follows from a direct calculation using \cref{hn}.  
\end{proof}
Now we are ready to deduce \cref{special} for the case corresponding to $p$-divisible groups over finite fields. 
\begin{proposition}\label{l3}
  Let $r \in \mathbb{Z}$. Let $k= \mathbb{F}_q$, where $q= p^n$. Suppose that $M$ is a vector bundle on $(\mathrm{Spec}\, k)^{\mathrm{syn}}$ with Hodge--Tate weights in $\left \{0,1\right \}$. Suppose that $\rho$ is the order of the zero of $\zeta (M, s).$ Then 
$$\left| \lim_{s \to r} \frac{\zeta (M,s)}{(1-q^{r-s})^\rho}\right| _p = \frac{1}{\mu_{\mathrm{syn}}(M,r)q^{\chi(M,r)}}.$$   
\end{proposition}

\begin{proof}
 Let $K= W(k) [1/p]$ and $M_K$ denote the isocrystal associated to $M$. Suppose that $G$ is the $p$-divisible group that corresponds to $M$. Let $(u_i)_{i \in I}$ be the set of roots of $\det (1- tF^n|M_K)$ counted with multiplicity. By the Dieudonn\'e--Manin classification, it follows that $\lambda_i:= v_q(u_i) = v_p(u_i)/n$ is the set of slopes of $M_K$ counted with multiplicity. By definition of the zeta function, it follows that 
\begin{equation}\label{eq11}
  \left| \lim_{s \to r} \frac{\zeta (M,s)}{(1-q^{r-s})^\rho}\right|^{-1} _p = \prod_{u_i \ne q^r} |1-u_i q^{-r}|_p . 
\end{equation}

Now we will compute $\mu_{\mathrm{syn}}(M,r).$ Note that $H^{>0}(\overline{M}, \mathbb{Z}_p(r))=0$ by \cref{ts}. Let $\gamma$ denote the Galois action on $H^0 (\overline{M}, \mathbb{Z}_p(r)).$ Since $H^0 (\overline{M}, \mathbb{Z}_p(r))$ is a finitely generated $\mathbb{Z}_p$-module (\cref{ts}) it follows from \cref{uk} that $\mu_{\mathrm{syn}}(M, r) =  \prod_{\alpha_i \ne 1}|\alpha_i - 1|_p$ where $\alpha_i$ is the set of eigenvalues of $\gamma$ on $H^0 (\overline{M}, \mathbb{Q}_p(r)).$ Note that $H^0 (\overline{M}, \mathbb{Q}_p(r)) \simeq (\overline{M}[1/p])^{\overline{F}=p^r}.$ By the Dieudonn\'e--Manin classification, the set of eigenvalues $\alpha_i$ is the same as the set $q^r/u_i$ such that $v_p(q^r/u_i)=0.$ Therefore, $$\mu_{\mathrm{syn}} (M, r) = \prod_{\substack{u_i \ne q^r \\ v_p(q^r/u_i)=0}} |1-q^r/u_i|_p=\prod_{\substack{u_i \ne q^r \\ v_p(q^r/u_i)=0}} |1-u_i q^{-r}|_p. $$ Since $|1- u_iq^{-1}|_p=1$ when $v_p(u_iq^{-r})>0$, we have
\begin{equation}
 \prod_{u_i \ne q^r} |1-u_i q^{-r}|_p = \mu_{\mathrm{syn}} (M,r) \prod_{\substack{u_i \ne q^r\\ v_p (u_iq^{-r})<0}}  |1-u_iq^{-r}|_p. 
\end{equation}
Note that $  |1-u_iq^{-r}|_p =  |1-u_iq^{-r}|_q.$ Combining the above equalities, we see that in order to prove the proposition, it would now suffice to show that 
$$ \prod_{\substack{u_i \ne q^r\\ v_p (u_iq^{-r})<0}}  |1-u_iq^{-r}|_q = q^{\chi(M,r)}.$$ This is equivalent to showing that \begin{equation}\label{eq5}
 \sum_{\substack{i,\,v_q(u_i) < r}}r - v_q(u_i) = \chi(M,r).
\end{equation}To this end, note that the left hand side in \cref{eq5} is equal to $\sum_{i, \, \lambda_i \le r} (r- \lambda_i)$. Since $M_K$ must have slope in $[0,1]$, it follows that the sum is zero when $r \le 0$, which gives the claim in that case by \cref{usefulll}. For $r>0$, we have $\sum_{i, \, \lambda_i \le r} (r- \lambda_i) = r \cdot \mathrm{height} (G) - \sum \lambda_i $. Since $\mathrm{height}(G) = \dim G + \dim G^\vee$, by using \cref{coconut}, it follows that 
$$r \cdot \mathrm{height} (G) - \sum \lambda_i = r \cdot \dim G^\vee + (r-1)\cdot \dim G.$$ Applying \cref{usefulll} again finishes the proof.
\end{proof}
Now we can also deduce \cref{special} in the general case.
\begin{proof}[Proof of \cref
{special}] By \cref{l0} the result is true for dualizable $F$-gauges for which a power of $p$ is null-homotopic. Therefore, by \cref{l1}, we can reduce to the case when $M$ is a vector bundle. By \cref{l2}, we may further reduce to the case when $M$ is a vector bundle with Hodge--Tate weights in $\left \{i, i+1\right \}$ for some integer $i$. By applying a suitable Breuil--Kisin twist and using \cref{twistsz}, we can further reduce to the case when $i=0$. In this case, the result follows from \cref{l3}. This finishes the proof.
\end{proof}

\section{Surfaces over finite fields}\label{surface}In this section, we discuss some consequences of our work in the case of surfaces. As before, let $k= \mathbb{F}_q$, where $q=p^n$. Let $X$ be a smooth, proper, geometrically connected surface over $k$.
\begin{definition}
    We let $$\mu_{\mathrm{syn}}(X, 1):=\mu_{\mathrm{syn}}(M(X), 1),$$ where $M(X)$ is the dualizable object in $F\text{-}\mathrm{Gauge}(k)$ associated to $X$ (see \cref{katz}). 
\end{definition}{}

 By definition, $\mu_{\mathrm{syn}}(X, 1)$ is a rational number and is defined unconditionally. Let $\mathrm{NS}(X)$ be the N\'eron--Severi group of $X$, which is a finitely generated abelian group equipped with a non-degenerate pairing. Suppose that $(D_i)_{i}$ is a basis of $\mathrm{NS}(X)$ modulo torsion. Let $\mathrm{det} (D_i \cdot D_j)$ denote the determinant of the matrix obtained by using the pairing. Let $|\mathrm{NS}_{\mathrm{tors}}(X)|$ denote the torsion subgroup of $\mathrm{NS}(X)$. 

\begin{definition}
In the above set up, we define

$$\beta r(X)_p:= \frac{\mu_{\mathrm{syn}}(X,1)|[\mathrm{NS}_{\mathrm{tors}}(X)]^2|_p}{|\mathrm{det}(D_i \cdot D_j)|_p}.$$    
\end{definition}

The above quantity is defined unconditionally. We will show that if one assumes that the Artin--Tate conjecture (in particular, this implies that the Brauer group $\mathrm{Br}(X)$ is finite) holds, then $\beta r(X)_p =|\mathrm{Br}(X)|_p,$ where the latter denoted $p$-adic norm of the Brauer group. To this end, let us recall the Artin--Tate conjecture. Note that the zeta function for $X$ can be expressed in the following form.
$$\zeta (X, s) =  \frac{P_1(X,q^{-s})P_1(X, q^{1-s})}{(1-q^{-s})P_2(X, q^{-s})(1-q^{2-s})},$$ where $P_i(X,T)$ is defined in terms of characteristic polynomial of (suitable power) the Frobenius on $i$-th $\ell$-adic (or crystalline) cohomology. The Artin--Tate conjecture says that $\mathrm{Br}(X)$ is finite, and 
$$\lim_{s \to 1}\frac{P_2 (X, q^{-s})}{(1-q^{1-s})^{\rho(X)}} = \frac{[\mathrm{Br}(X)]\cdot |\det (D_i \cdot D_j)|}{q^{\alpha(X)}[\mathrm{NS}(X)_{\mathrm{tors}}]^2},$$ where $\rho(X)$ is the rank of $\mathrm{NS}(X)$ and $\alpha(X) = \chi(X, \mathcal{O}_X) - 1 + \dim (\mathrm{Pic Var}(X)).$

\begin{proposition}\label{compar}
Assume that the Artin--Tate conjecture holds for the surface $X$. Then $$\beta r (X)_p = |\mathrm{Br}(X)|_p.$$    
\end{proposition}

\begin{proof}
As a consequence of \cref{special}, unconditionally, we have 
\begin{equation}\label{sir}
 \left| \lim_{s \to 1} \frac{\zeta (X,s)}{(1-q^{1-s})^\rho}\right| _p = \frac{1}{\mu_{\mathrm{syn}}(X,1)q^{\chi(X,1)}}.\end{equation}{} From \cref{whe}, it follows that $\chi(X,1) = \chi(X,\mathcal{O}). $  

We will now obtain another expression for the left hand side assuming the Artin--Tate conjecture.
Note that the set of roots of $P_1(X,t)$ are inverse to the set of eigenvalues of Frobenius on $\ell$-adic cohomology. We will denote the latter set by $\alpha_i$. For each $i$, $\alpha_i$ has absolute value $q^{1/2}.$ In particular, $P_1(X, q^{-s})$ and $P_1 (X, q^{1-s})$ have no zero (or poles) at $s=1$. Moreover, it follows that
$$|P_1 (X, 1)|_p = \prod_i |1- \alpha_i|_p =1.$$
On the other hand, 
$$P_1 (X, q^{-1})|_p = |1/q^{2 \dim (\mathrm{PicVar}(X))}|_p \prod_i |q- \alpha_i|_p =  q^{\dim \mathrm{PicVar}(X)}.$$
By the Artin--Tate conjecture, $\zeta (X, s)$ has a pole at $s=1$ of order $\rho(X)$. Therefore, $\rho = - \rho(X)$ in the equation \cref{sir}. Further, we must have
\begin{equation}
  \left| \lim_{s \to 1} \frac{\zeta (X,s)}{(1-q^{1-s})^\rho}\right| _p =  \frac{q^{\dim \mathrm{PicVar}(X)}\cdot|[\mathrm
  NS(X)_{\mathrm{tors}}]^2|_p}{q \cdot |\mathrm{Br}(X)|_p\cdot |\mathrm{det}(D_i \cdot D_j)|_p \cdot q^{\alpha(X)}}.  
\end{equation}
The right hand side simplifies to $\frac{|[\mathrm
  NS(X)_{\mathrm{tors}}]^2|_p}{|\mathrm{Br}(X)|_p\cdot |\mathrm{det}(D_i \cdot D_j)|_p \cdot q^{\chi(X, \mathcal{O})}}.$
Combining with \cref{sir}, we obtain 
$$\frac{1}{\mu_{\mathrm{syn}}(X,1)q^{\chi(X,\mathcal{O})}}= \frac{|[\mathrm
  NS(X)_{\mathrm{tors}}]^2|_p}{|\mathrm{Br}(X)|_p\cdot |\mathrm{det}(D_i \cdot D_j)|_p \cdot q^{\chi(X, \mathcal{O})}}. $$ This implies the claim.
\end{proof}

\begin{question}\label{qn}
Let $X$ be a surface as before. Is the number $\frac{\mu_{\mathrm{syn}}(X,1)}{|\mathrm{det}(D_i \cdot D_j)|_p}$ square of a rational number? 
\end{question}{}

Note that the numbers above are defined unconditionally. Below, we deduce an affirmative answer by assuming the Artin--Tate conjecture. 
\begin{proposition}\label{compare2}
Let $X$ be a surface over a finite field as before and let $(D_i)_i$ denote a basis of $\mathrm{NS}(X)$ modulo torsion. Suppose that the Artin--Tate conjecture holds. Then $$\frac{\mu_{\mathrm{syn}}(X,1)}{|\mathrm{det}(D_i \cdot D_j)|_p}$$ is a square.  
\end{proposition}{}
\begin{proof}
 Under the assumptions, the Brauer group $\mathrm{Br}(X)$ is finite and is a square \cite{brauer}. Thus $|\mathrm{Br}(X)|_p$ is a square. The claim now follows from \cref{compar}.   
\end{proof}

\begin{remark}
We expect \cref{qn} to have an affirmative answer unconditionally. It is possible to modify the definition of $\mu_{\mathrm{syn}}(X, 1)$ to obtain an $\ell$-adic analogue for every prime $\ell \ne p$ and ask an analogous question. An affirmative answer to these questions should also give a new approach to showing that if $\mathrm{Br}(X)$ is finite, then $|\mathrm{Br}(X)|$ must be a square.
\end{remark}{}

\begin{remark}
 In odd characteristic, by Milne's work \cite{Mana}, finiteness of $\mathrm{Br}(X)$ implies the Artin--Tate conjectures for $X$. Therefore, in odd characteristic, \cref{compar} and \cref{compare2} holds only under the assumption that $\mathrm{Br}(X)$ is finite.
\end{remark}{}

\bibliographystyle{amsalpha}
\bibliography{main}
\end{document}